\documentclass[12pt]{article}
\usepackage[letterpaper,margin=1.9cm]{geometry}
\usepackage{latexsym, amsfonts, amsthm,amssymb,amscd,amsmath,makeidx}
\usepackage{enumerate}
\usepackage{color}
\usepackage{graphicx,psfrag,epsf}
\usepackage{natbib}
\usepackage{url} 
\usepackage[normalem]{ulem}

\usepackage[permil]{overpic}
\newcommand{\blind}{1}

\definecolor{DarkGreen}{rgb}{0.2,0.6,0.2}
\def\black#1{\textcolor{black}{#1}}

\usepackage{bbm}
\def\argmin{\mathop{\rm arg\,min}}

\usepackage{algorithm}
\usepackage[noend]{algpseudocode}

\usepackage{caption}
\usepackage{subcaption}

\newtheorem{lemma}{Lemma}[section]
\newtheorem{theorem}{Theorem}[section]
\newtheorem{proposition}{Proposition}[section]

\theoremstyle{definition}

\newtheorem{example}{Example}[section]
\theoremstyle{remark}
\newtheorem{remark}{Remark}[section]

\usepackage[toc,page]{appendix}

\allowdisplaybreaks

\usepackage[bottom,flushmargin,hang,multiple]{footmisc}
\usepackage{setspace}
\usepackage{xr-hyper}
\usepackage{hyperref}
\usepackage{commath}
\definecolor{darkblue}{rgb}{0.1,0.1,0.9}
\definecolor{darkred}{rgb}{0.9,0.1,0.1}

\hypersetup{colorlinks, allcolors= darkblue}

\begin{document}

\def\spacingset#1{\renewcommand{\baselinestretch}%
{#1}\small\normalsize} \spacingset{1}


\title{\bf The Multi-Armed Bandit Problem Under the Mean-Variance Setting}
  
\author{
\normalsize Hongda Hu\thanks{Department of Statistics and Actuarial Science, University of Waterloo}
\and 
\normalsize Arthur Charpentier\thanks{Universit\'e du Qu\'ebec \`a Montr\'eal}
\and 
\normalsize Mario Ghossoub\thanks{Department of Statistics and Actuarial Science, University of Waterloo}
\and 
\normalsize Alexander Schied\thanks{Department of Statistics and Actuarial Science, University of Waterloo, {\tt aschied@uwaterloo.ca}\hfill\break 
  We are grateful to three anonymous reviewers for comments that have greatly helped us to improve this paper. H.H.~and A.S.~gratefully acknowledge support from the Natural Sciences and Engineering Research Council of Canada through grant RGPIN-2017-04054. H.H.~and M.G.~gratefully acknowledges support from the Natural Sciences and Engineering Research Council of Canada through grant RGPIN-2018-03961. A.C.~acknowledges the financial support of the AXA Research Fund through the joint research initiative {\em use and value of unusual data in actuarial science}, as well as from the Natural Sciences and Engineering Research Council of Canada through grant RGPIN-2019-07077. 
  }
  }
  
  \date{\normalsize 
  First version:   December 18, 2022\\
  This version: May 2, 2024 }
     \maketitle

\if0\blind
{
  \bigskip
  \bigskip
  \bigskip
  \begin{center}
    {\LARGE\bf Title}
\end{center}
  \medskip
} \fi


\bigskip
\begin{abstract}
The classical multi-armed bandit problem involves a learner and a collection of arms with unknown reward distributions. At each round, the learner selects an arm and receives new information. The learner faces a tradeoff between exploiting the current information and exploring all arms. The objective is to maximize the expected cumulative reward over all rounds. Such an objective does not involve a risk-reward tradeoff, which is fundamental in many areas of application. In this paper, we build upon \citet{sani2012risk}’s extension of the classical problem to a mean-variance setting. We relax their assumptions of independent arms and bounded rewards, and we consider sub-Gaussian arms. We introduce the Risk-Aware Lower Confidence Bound algorithm to solve the problem, and study some of its properties. We perform numerical simulations to demonstrate that, in both independent and dependent scenarios, our approach outperforms the algorithm suggested by \citet{sani2012risk}

\end{abstract}




\section{Introduction} \label{section1}

The multi-armed bandit (MAB) problem is a type of online learning and sequential decision-making problem. A classical MAB problem involves $K$ independent arms, each with its own independent reward distribution, and a learner. In the problem, each arm generates a random reward from the underlying probability distribution, which is unknown in advance. At each round, the learner selects one arm among $K$ arms and receives a new observation from that arm. The learner often faces a tradeoff between exploiting the current information and exploring all arms. Given a finite number of $n$ rounds, the design objective aims to maximize the cumulative reward over $n$ rounds.\par

The MAB framework has been applied to a wide range of real-world problems, including clinical trials (\citet{durand2018contextual}), recommendation systems (\citet{mary2015bandits}), telecommunication (\citet{boldrini2018mumab}) and finance (\citet{shen2015portfolio}). For example, in finance (\citet{shen2015portfolio}), the actions correspond to the proposed portfolio composition, and the reward represents the performance of the proposed portfolio. The exploration–exploitation dilemma in the above example is whether to try a new portfolio composition or keep using the current portfolio with the best historical performance. \par

As is standard in the literature, we measure the welfare implications of a given selection policy by the notion of regret. The latter is defined as the difference between the expected cumulative rewards under the true distribution and the empirical distribution.\par

\subsection{Related Work}

The MAB problem has been explored by \citet{thompson1933likelihood} and \citet{robbins1952some} as a useful tool for constructing online sequential decision algorithms. \citet{thompson1933likelihood} proposed a Bayes-optimal approach that directly maximizes expected cumulative rewards with respect to a given prior distribution. \citet{robbins1952some} first formalized the classic MAB problem, in which each arm follows an unknown distribution over [0,1], and rewards are independent draws from the distribution corresponding to the chosen arm. Since then, a number of methods have been developed to solve the MAB problem. Most of the possible ways to solve the MAB problem can be defined into three categories:

\begin{itemize}

\item \textit{$\varepsilon$-greedy algorithm:} is a straightforward method for balancing exploration and exploitation by employing a greedy policy to take action with a probability of $1-\varepsilon$, and a random action with a small probability of $\varepsilon$.

\item \textit{Upper confidence bounds (UCB) algorithm:} is often phrased as ``optimism in the face of uncertainty". The algorithm uses the observed data so far to assign a value (i.e., upper confidence bound index) to each arm, where the arm with the higher index value will be selected with a higher probability in the next round. Strong theoretical guarantees on the rate of regret can be attained by implementing the UCB algorithm. Early work on using UCB to solve MAB problems was done by \citet{lai1985asymptotically}, \citet{agrawal1995sample} and \citet{auer2002finite}. The technique of upper confidence bounds (UCB) for the asymptotic analysis of regret was introduced by \citet{lai1985asymptotically}, who demonstrated that the minimum regret has a logarithmic order in $n$. As a simpler formulation, \citet{agrawal1995sample} discusses the MAB problem under a finite-time setting. \citet{auer2002finite} conduct a finite-time analysis of the classic MAB with bounded rewards using the UCB algorithm.

 \item \textit{Thompson sampling:} is the first algorithm for the bandit problem proposed by \citet{thompson1933likelihood}. Thompson sampling has a simple idea. By implementing Thompson sampling, the learner selects a prior over a set of possible bandit environments at the start. In each round, the learner makes an action chosen randomly from the posterior. Thompson Sampling has been empirically proved to be effective by \citet{chapelle2011empirical}. In comparison to the UCB algorithm, Thompson Sampling does not have strong theoretical guarantees on regret. \black{\citet{zhu2020thompson}, \cite{baudry2021optimal}, and \citet{chang2022unifying} conducted theoretical analysis and established consistent guarantees on regret.}
\end{itemize}

In the classic MAB formulation, the learner maximizes the expected cumulative reward within a finite time frame. Nevertheless, in real practice, maximizing the expected reward may not be the ultimate goal. For instance, in portfolio selection, the portfolio that performs best during a period of the economic boom may lead to a substantial loss when the economy is struggling. In clinical trials, the treatment that works best on average may result in side effects for some patients. In many cases, a learner's primary goal may be to effectively balance risk and return. There is no universally accepted definition of risk. Among various risk modelling paradigms, the mean-variance paradigm (\citet{markowitz1968portfolio}) and the expected utility theory (\citet{morgenstern1953theory}) are the two fundamental risk modelling paradigms. 

The work of \citet{sani2012risk}, in which the criteria of mean-variance of data was first introduced, is the most important reference to our study. \citet{sani2012risk} study the problem where each arm's distribution is bounded, and the learner’s objective is to minimize the mean-variance of rewards collected over a finite time. Their proposed Mean-Variance Lower Confidence Bound (MVLCB) algorithm achieves a learning regret of $\mathcal{O}((\log(n))^2/n)$. As an improvement, \citet{vakili2016risk} analyze the problem under the mean-variance setting, but with the assumption that the square of the underlying $X$ is sub-Gaussian. \citet{vakili2016risk}'s algorithm achieves a learning regret of $\mathcal{O}((\log(n))/n)$. Furthermore, \citet{vakili2015mean} extend the previous setup by taking into account the value-at-risk ($\textup{VaR}$) of total rewards over a finite time horizon and ending with a learning regret of $\mathcal{O}((f(n)\log(n))/n)$, where $f(n)$ is a positive increasing diverging sequence. \citet{zhu2020thompson} examine the possibilities of Thompson sampling methods for solving mean-variance MAB and provide a detailed regret analysis for Gaussian and Bernoulli bandits with comparable learning regret $\mathcal{O}((\log(n))/n)$. In addition, they run a large number of simulations to compare their algorithms against existing UCB-based methods for different risk tolerances. As another way of measuring the risk, conditional-value-at-risk ($\textup{CVaR}$) has been studied in solving risk-averse MAB by \citet{galichet2013exploration}, \citet{kagrecha2019distribution} and \citet{prashanth2020concentration}. \citet{galichet2013exploration} show that their proposed algorithm's learning regret under the $\textup{CVaR}$ scheme is $\mathcal{O}((\log(n))/n)$. As a follow-up, \citet{kagrecha2019distribution} conduct a theoretical study of $\textup{CVaR}$-based algorithm for MAB with unbounded rewards, and \citet{prashanth2020concentration} compare the performance of $\textup{CVaR}$-based algorithm under light-tailed and heavy-tailed distributions.\par

Both \citet{sani2012risk} and \citet{galichet2013exploration} study the risk-aware MAB problem by assuming independent bounded arms. Such an assumption limits the application of their results. In real-life problems, multiple arms may expose the same source of risk, and the distribution of each arm may not necessarily be bounded. As a pioneer, \citet{liu2020risk} create a risk-aware UCB algorithm for Gaussian distributions that has a learning regret of $\mathcal{O}((\log(n))/n)$. The Thompson sampling methods are then used by  \citet{gupta2021multi} to analyze the correlated bandit problems, where the underlying distributions are bounded. They categorize the arms as competitive or non-competitive, assuming that correlations between the arms arise from a latent random source. For the competitive group, their algorithm produces a learning regret of $\mathcal{O}((\log(n))/n)$, whereas for the non-competitive group, it can only produce $\mathcal{O}(1)$.\par

Bandit learning algorithms are now widely used in a variety of fields. However, the application of bandit learning to finance has received little attention from researchers due to the inherent difference between financial assets and classic bandits.
\begin{itemize}

    \item \textit{Independence:} the classic MAB problem assumes i.i.d.~reward distributions for each arm, which does not apply to financial asset returns. Assets are correlated in the financial markets, and the performance of one asset can reveal important information about the other assets. A lower regret bound can be obtained by exploiting the dependence; for example, \citet{gupta2021multi} pioneered the analysis of the correlated MAB problem.

    \item\textit{Single pull:} the classic MAB problem only allows one arm to be pulled at each round, while portfolio selection strategies often result in a basket of financial assets for investment.
    
    \item \textit{Risk-return balance:} the classic MAB problem is only concerned with maximizing mean reward, while good portfolio selection strategies focus on risk-adjusted return (i.e., Sharpe ratio, mean-variance).
    \item \textit{Historical data:} the classic MAB problem has no historical data, whereas sufficient historical data is available for publicly traded financial assets.
    
\end{itemize}
\citet{shen2015portfolio} use a principal component decomposition (PCA) to construct an orthogonal portfolio from correlated assets, and they derive the optimal portfolio strategy using the UCB framework based on risk-adjusted reward. Later, \citet{shen2015portfolio} treat classic portfolio strategies in finance as strategic arms, and they leverage the Thompson sampling method to mix two different strategic arms.

\subsection{Contributions}
\black{We analyze the risk-aware Multi-Armed Bandit (MAB) problem using the Upper Confidence Bound (UCB) methodology proposed by \cite{lai1985asymptotically}. In contrast to their work, we factor in risk by adopting a mean-variance framework. Shifting from a sole focus on rewards to a consideration of risk-reward balance, our algorithm can be applied to a broader range of problems compared to the UCB algorithm proposed by \cite{lai1985asymptotically}. Comparing our methodology to another well-known MAB method proposed by \cite{thompson1933likelihood}, we attain mathematically rigorous bounds on the realized regret. In contrast, much of the literature that adopts the Thompson sampling method can demonstrate the effectiveness of the algorithms only empirically.}\par

\black{Next, \cite{sani2012risk} serves as our primary reference, being the first to introduce the mean-variance MAB problem. In their paper, they assume i.i.d.~arms taking values in $[0,1]$ and state a learning regret of $\mathcal{O}((\log(n))^2/n)$. Note, however, that the algorithm in  \cite{sani2012risk}  requires a fixed runtime $n$. That is, if the algorithm is set up to run for $m$ time periods, then its parameters, and consequently its policy, will  differ from those prescribed for a runtime of $n\neq m$. This fact somewhat limits the interpretation of the asymptotics for the learning regret. By contrast, the algorithm we introduce in this paper does not require that the runtime $n$ is fixed from the start. This algorithm is a member of the extensive family of Lower Confidence Bound algorithms, and we call it the \emph{Risk Aware Lower Confidence Bound (RALCB) algorithm}. In comparison to  \cite{sani2012risk},  we also allow for a larger class of arms distributions, namely the  sub-Gaussian distributions, and we do not the individual arms to be independent.  In addition, we are able to bound the expected regret by $\mathcal{O}((\log(n))/n)$.}\par

\black{Subsequent to our mathematical analysis, we perform numerous numerical simulations to assess and compare the performance of our algorithm against the one proposed by \cite{sani2012risk}. Across diverse risk-return scenarios, our algorithm rapidly identifies optimal arms, consistently surpassing the performance of the MVLCB algorithm suggested by \cite{sani2012risk}. Moreover, we illustrate the robustness of our algorithm by applying it to real-world data taken from financial markets.}\par

The rest of this paper is organized as follows: In Section \ref{section2}, we introduce the MAB problem under the mean-variance setting. In Section \ref{section3}, the confidence-bound algorithm is introduced, and its theoretical characteristics are examined in Section \ref{section4}. A possible extension of MAB problems is examined in Section \ref{section5}. The theoretical findings are then supported by a set of numerical simulations in Section \ref{section6}. Results of supporting proofs are provided in Section \ref{section7}.


\section{Problem Formulation}\label{section2}

In this section, we introduce notation and define the risk-aware MAB problem. The classic $K$-armed bandit problem has a finite action set $\mathcal{A}:=\{1,...,K\}$. The learner interacts with the environment sequentially over time. \black{Let $X^{i}_{t}$ denote the $t$-th observed real-valued reward from pulling arm $i\in \{1,...,K\}$}, modeled as a random variable on a given probability space $(\Omega,\mathcal{F},\mathbb{P})$. We assume that each $X^{i}_{t}$ belongs to the class of sub-Gaussian random variables. That is, there is a constant $\theta \ge 0$ such that for all $\lambda\in\mathbb{R}$: 
\begin{equation}\label{sub-Gaussian ieq}
\mathbb{E}_{\mathbb{P}}[e^{\lambda(X-\mathbb{E}_{\mathbb{P}}[X])}]\leq e^{\frac{\lambda^2 \theta^2 }{2}}.
\end{equation}
The parameter $\theta$ will henceforth be called a \emph{sub-Gaussianity parameter} for $X$. Clearly, all Gaussian random variables are sub-Gaussian, and so are all bounded random variables. We refer to \citet{wainwright2019high} for a discussion of the mathematical properties of sub-Gaussian random variables.

We let $\textbf{X}_t:=(X^{1}_{t},X^{2}_{t},...,X^{K}_{t})$ and assume that the random vectors $\textbf{X}_1,\textbf{X}_2,\dots$ are independent and identically distributed (i.i.d.), while for each fixed $t$, the random variables $X^{1}_{t},X^{2}_{t},...,X^{K}_{t}$ may be dependent. In each round $t$, the learner selects an arm $\pi_t\in \mathcal{A}$ based on the information available at $t$. This will be formalized by defining a \emph{learning policy} as a sequence $\pi_1,\pi_2,\dots$ of random variables  taking values in $\mathcal{A}=\{1,\dots,K\}$, such that $\pi_t$ depends only on $(X^{\pi_1}_{T_{\pi_1,1}},\dots,X^{\pi_{t-1}}_{T_{\pi_{t-1},t-1}},\pi_1,\dots,\pi_{t-1})$, where 
$$T_{i,t}:=\sum_{s=1}^{t} \mathbbm{1}_{\{\pi_s=i\}}$$
counts the number of times arm $i$ has been pulled by time $t$. The environment then samples a reward $X^{\pi_t}_{T_{\pi_t,t}}$ and reveals it to the learner. As a result, each decision of the learner is based on past observations.\par

In the classic MAB problem, the goal of the learner is to maximize the expected cumulative reward. In this paper, we also take the risk into consideration. To effectively balance the expected reward and the risk, we use the same mean-variance setting as in \citet{sani2012risk}. Let the coefficient of absolute risk tolerance $\rho \in [0,1]$ be given and fixed, and define for every arm $i$,
\begin{equation}\label{MV eq}
\black{\textup{MV}^{\rho}_i:=(1-\rho)\sigma_i^2-\rho \mu_i,}
\end{equation}
where $\mu_i$ and $\sigma_i^2$ denote the mean and variance of $X_i$. The optimal arms are given by $\argmin_{i\in \mathcal{A}} \textup{MV}^{\rho}_i$, and we will often use the notation $i^{\rho}_0$ to denote an optimal arm.

\black{Given a learning policy $\pi_t$ and its corresponding performance over $n$ rounds, the empirical mean-variance of the learning policy by round $n$ can be defined as}
\begin{equation}
  \black{\widehat{\textup{MV}}^{\rho}_n(\pi)}:=(1-\rho)\hat{\sigma}^2_n(\pi)-\rho \hat{\mu}_n(\pi),  \label{Eq_1}
\end{equation}
where
\begin{equation}\label{Eq_2}
    \black{\hat{\mu}_n(\pi):=\frac{1}{n}\sum_{i=1}^{K}\sum_{t=1}^{T_{i,n}} X^{i}_{t}}, \;\;\;\;\; \black{\hat{\sigma}^2_n(\pi):=\frac{1}{n}\sum_{i=1}^{K}\sum_{t=1}^{T_{i,n}}(X^{i}_{t}-\hat{\mu}_n(\pi))^2}.
\end{equation}

The objective of the risk-aware MAB problem is to find a policy $\pi$ that asymptotically minimizes  $\widehat{\textup{MV}}^{\rho}_n(\pi)$ for a given risk-tolerance factor $\rho$.

To compare the performance of different learning policies $\pi$ over $n$ rounds, we introduce the learning regret:
\begin{equation}
    \black{\mathcal{R}_n(\pi)}:=\widehat{\textup{MV}}^{\rho}_n(\pi)-\min_{i\in \mathcal{A}} \textup{MV}^{\rho}_i, \label{Eq_5}
\end{equation} 
which is the difference between the policy's empirical mean-variance and the optimal mean-variance. As a consequence, the objective becomes to generate an algorithm for which the regret decreases in expectation as $n$ increases. 

\begin{remark}\label{Rem_1} 
\black{By examining the extreme values of risk tolerance, we observe that the minimization of the learning regret \eqref{Eq_5} includes the following two special cases.
\begin{itemize}
    \item As $\rho =1$, the problem reduces to the standard expected reward maximization problem. 
    \item For $\rho =0 $, the problem transforms into a variance minimization problem.
\end{itemize}}
\end{remark}


\section{RALCB Algorithm}\label{section3}

In this section, we introduce our Risk-Aware Lower-Confidence Bound Algorithm (RALCB) and state our main results. To this end, we fix a parameter $\theta>0$ that is  a common sub-Gaussianity parameter for all arms. That is,  
\black{\begin{equation}\label{theta for all arms eq}
\mathbb{E}_{\mathbb{P}}[e^{\lambda(X_i-\mathbb{E}_{\mathbb{P}}[X_i])}]\leq e^{\frac{\lambda^2 \theta^2 }{2}}
\end{equation}}
 holds for all $\lambda \in \mathbbm{R}$ and $i \in \mathcal{A}$. In typical applications of the multi-armed bandit problem, the exact distribution of the arms will not be known in advance. We therefore emphasize that the parameter $\theta$ only needs to be some number such that \eqref{theta for all arms eq} holds.  In many situations, it will be possible to infer such an a priori bound by means of historical data or by functional bounds on the ranges of certain arms. We will see below that our bounds on the regret will improve when we are able to take a smaller $\theta$, with the optimal lower bound given by  
 
\black{\begin{equation}
    \theta^*:=\inf \left\{\theta: \text { for all } \lambda \in \mathbb{R} \text { and } i \in \mathcal{A}, \mathbb{E}_{\mathbb{P}}\left[e^{\lambda\left(X_i-\mathbb{E}_{\mathbb{P}}\left[X_i\right]\right)}\right] \leq e^{\frac{\lambda^2 \theta^2}{2}}\right\}.
\end{equation}}

For each arm $i$, we can define the empirical mean-variance with $t$ samples as
\begin{equation}
  \widehat{\textup{MV}}^{\rho}_{i,t}:=(1-\rho)\hat{\sigma}_{i,t}^2-\rho \hat{\mu}_{i,t},  \label{Eq_3}
\end{equation}
where
\begin{equation}
    \black{\hat{\mu}_{i,t}:=\frac{1}{t}\sum_{s=1}^{t} X^{i}_{s}, \;\;\;\;\; \hat{\sigma}_{i,t}^2:=\frac{1}{t}\sum_{s=1}^{t}(X^{i}_{s}-\hat{\mu}_{i,t})^2} . \label{Eq_4}
\end{equation}
We furthermore define 
$$\varphi(x):=32(1-\rho)\theta^2 \max\left(\sqrt{x/2},x\right)+(1-\rho)\theta^2 x+ \rho\theta \sqrt{x}$$
and  
\begin{equation}
    V^{\textup{RALCB}}_{i, t-1}:=\widehat{\mathrm{MV}}^{\rho}_{i, t-1}-\varphi \left(\frac{8\log (t)}{T_{i,t-1}}\right).\label{Eq_6}
\end{equation}
\black{In comparison to the Mean-Variance Lower Confidence Bound Algorithm (MVLCB) introduced by \cite{sani2012risk}, 
\begin{equation}
    V^{\textup{MVLCB}}_{i, t-1}=\widehat{\mathrm{MV}}^{\rho}_{i, t-1}-(5+\rho)\sqrt{\frac{\log (1/\delta)}{T_{i,t-1}}},\label{MVLCB1}
\end{equation}
our approach shares the same exploitation term but diverges in the exploration term. First, the exploration term of MVLCB is characterized by a parameter $\delta$, depending upon the time horizon $n$. In contrast, our algorithm adopts a dynamic approach by updating the exploration term at each time step $t$. This adaptation enables us to discuss the asymptotic performance of the regret. Second, with a precise $\varphi$ function, our algorithm can handle all sub-Gaussian random variables, whereas the exploration term in the MVLCB algorithm is only designed for bounded random variables taking values in $[0,1]$.} \par

\black{Furthermore, \cite{zhao2019multi} introduced a relevant algorithm addressing the mean-variance Multi-Armed Bandit (MAB) problem, termed the Mean-Variance Upper Confidence Bound Algorithm (MVUCB), under more restrictive underlying assumptions. It is based on the following index,
\begin{equation}
    V^{\textup{MVUCB}}_{i, t-1}=\widehat{\mathrm{MV}}^{\rho}_{i, t-1}-b\sqrt{\frac{ \log (t)}{T_{i,t-1}}}.\label{MVLCB1}
\end{equation}
In their framework, the exploration term is updated over time, but they do not provide an explicit formulation for the policy parameter $b$.}

\begin{algorithm}[H]
\caption{Risk-Aware Lower-Confidence Bound Algorithm (RALCB)}\label{alg:ralcb}
\footnotesize
\begin{algorithmic}[1]
\State \textbf{Input:} \black{$\theta$, $K$, $\rho$}.
\State \textbf{for each} $t=1,2,...,K$ \textbf{do} 
\State \qquad Play arm $\pi_t=t$ and observe $X^{\pi_t}_1$.
\State \textbf{end for}
\State Update $\widehat{\textup{MV}}^{\rho}_{i,1}$ for $i=1,...,K$ by Equation (\ref{Eq_3}).
\State Set $T_{i,K}=1$ for $i=1,...,K$.

\State \textbf{for each} \black{$t=K+1,K+2,...$} \textbf{do} 
\State \qquad \textbf{for each} $i=1,2,...,K$ \textbf{do} 

\State \qquad \qquad Compute $V^{\textup{RALCB}}_{i, T_{i,t-1}}=\widehat{\mathrm{MV}}^{\rho}_{i, T_{i,t-1}}-\varphi \left(\frac{8\log (t)}{T_{i,t-1}}\right).$
\State \qquad \textbf{end for}
\State \qquad Return $\pi_{t}=\argmin _{i=1,...,K} V^{\textup{RALCB}}_{i, T_{i,t-1}}.$
\State \qquad Update $T_{i,t}=T_{i,t-1}+1$ for $i=1,...,K$.
\State \qquad Observe $X^{\pi_t}_{T_{i,t}}.$
\State \qquad Update $\widehat{\textup{MV}}^{\rho}_{i,T_{i,t}}$ by Equation (\ref{Eq_3}).
\State \textbf{end for}
\end{algorithmic}
\end{algorithm}

Let us explain intuitively how the RALCB algorithm works. At each time $t$, the algorithm chooses an arm that \black{minimizes} the corresponding upper confidence index $V^{\textup{RALCB}}_{i, t-1}$. That is, the algorithm chooses the policy response 
$$\pi_{t}=\argmin _{i=1,...,K} V^{\textup{RALCB}}_{i, t-1}=\argmin_{i\in \mathcal{A}} \left(\widehat{\mathrm{MV}}^{\rho}_{i, t-1}-\varphi \left(\frac{8\log (t)}{T_{i,t-1}}\right)\right).$$
 Therefore, the arm $i$ is selected by the algorithm in either of the following two scenarios:
\begin{itemize}
    \item $\widehat{\mathrm{MV}}^{\rho}_{i, t-1}$ is small: the algorithm tends to exploit the best performer.
    \item $\varphi \left(\frac{8\log (t)}{T_{i,t-1}}\right)$ is large: the algorithm tends to explore alternative arms with insufficient observations.
\end{itemize}
\subsection{Regret Analysis}\label{section4}
\subsubsection{Upper Bound Analysis}
\begin{theorem}\label{Thm_1}
Suppose that $\mathcal{A}^{0}$ is the set of optimal arms given by $\argmin_{i\in \mathcal{A}} \textup{MV}^{\rho}_i$. The expected regret of the RALCB algorithm after $n$ rounds can be upper bounded as:
\begin{equation}
        \black{\mathbb{E}[\mathcal{R}_n(\pi)]}\leq\frac{1}{n}\sum_{i \in \mathcal{A}\setminus \mathcal{A}^{0}}\left(\frac{8\log(n)}{\varphi^{-1}(\Delta_i/2)}+5\right)(\Delta_{i}+2\Gamma_{i, \max }^{2})+\frac{5}{n}\sum_{i=1}^{K} \sigma_i^2, \label{Eq_16}
\end{equation}
where $\Delta_{i}:=\left(\sigma_{i}^{2}-\sigma_{i^{\rho}_0}^{2}\right)-\rho\left(\mu_{i}-\mu_{i^{\rho}_0}\right)$, $\Gamma_{i, \max }:=\max \left\{\left|\mu_{i}-\mu_{h}\right|: h=1, \ldots, K\right\}$, and 
\begin{equation*}
\varphi^{-1}(x) = \left\{
        \begin{array}{ll}
             \left(\frac{-(\rho\theta+16\sqrt{2}(1-\rho)\theta^2)+\sqrt{(\rho\theta+16\sqrt{2}(1-\rho)\theta^2)^2+4(1-\rho)\theta^2 x}}{2(1-\rho)\theta^2}\right)^2 & \quad x \in [0,\frac{17}{2}(1-\rho)\theta^2+\frac{\sqrt{2}}{2}\rho\theta), \\
            \left(\frac{-\rho\theta+\sqrt{\rho^2\theta^2+132(1-\rho)\theta^2x}}{66(1-\rho)\theta^2}\right)^2 & \quad x \in [\frac{17}{2}(1-\rho)\theta^2+\frac{\sqrt{2}}{2}\rho\theta,\infty),
        \end{array}
    \right.
\end{equation*}
is the inverse function of the function $\varphi$ defined in \eqref{theta for all arms eq}.
\end{theorem}

\black{The regret upper bound presented by \cite{sani2012risk} is a finite-time version due to its dependence on the confidence level $\delta$, which requires that the number $n$ of rounds is fixed from the beginning. As a result, their algorithm requires prior knowledge of $n$. In contrast, our paper adopts a dynamic approach by updating the confidence index at each time step $t$. This adaptation enables us to deal with scenarios where the horizon is not predetermined. It also allows us to discuss the asymptotic performance of the regret.}

\begin{remark}
Based on the results in Theorem \ref{Thm_1}, we can notice that the upper bound of the expected regret decreases as \black{$\mathcal{O}((\log (n))/n)$}. Taking $n\rightarrow \infty$, the expected regret of the RALCB algorithm satisfies:
\begin{equation}
        \black{\lim_{n \rightarrow \infty}\mathbb{E}[\mathcal{R}_n(\pi)]=0.} \label{Eq_2.8}
\end{equation}
According to Definition 4.2 in \citet{zhao2019multi}, the RALCB algorithm is hence a consistent policy. By substituting Equation (\ref{Eq_5}) into Equation (\ref{Eq_2.8}), we can obtain:
\begin{equation}
        \black{\lim_{n \rightarrow \infty}\mathbb{E}[\widehat{\textup{MV}}^{\rho}_n(\pi)]=\min_{i\in \mathcal{A}} \textup{MV}^{\rho}_i. }\label{Eq_2.9}
\end{equation}
According to Equation (\ref{Eq_2.9}), we can conclude that the expectation of the empirical mean-variance of the RALCB algorithm converges to the mean-variance of the optimal arms.
\end{remark}

\begin{remark}
\black{By examining the increasing function $\varphi(x)$, expressed as:
$$
\varphi(x)= \begin{cases}32(1-\rho)\theta^2\sqrt{x / 2}+\rho \theta \sqrt{x}+(1-\rho)\theta^2 x & x \in\left[0, \frac{1}{2}\right) \\ 32 (1-\rho)\theta^2 x+\rho \theta \sqrt{x} & x \in\left[\frac{1}{2}, \infty\right),\end{cases}
$$
it is evident that $\varphi(x)$ is a decreasing function of $\theta$ if $x$ is fixed. Thus, by substituting $\theta$ with a smaller $\tilde\theta$, the value $\varphi^{-1}(\Delta_i/2)$ increases, and in turn the upper bound \eqref{Eq_16}
  for the expected regret becomes tighter. The optimal bound is achieved for 
$$\theta^*:=\inf \left\{\theta: \text { for all } \lambda \in \mathbb{R} \text { and } i \in \mathcal{A}, \mathbb{E}_{\mathbb{P}}\left[e^{\lambda\left(X_i-\mathbb{E}_{\mathbb{P}}\left[X_i\right]\right)}\right] \leq e^{\frac{\lambda^2 \theta^2}{2}}\right\}.$$ }
\black{We emphasize once again, though,that in typical applications the exact distributions of the arm will be unknown, which highlights the fact that our algorithm works with any a priori bound $\theta\ge\theta^*$.}
\end{remark}

Moreover, we can derive a high probability upper bound for the regret of the RALCB algorithm.
\begin{theorem}\label{Thm_2}
\black{Suppose that $\mathcal{A}^{0}$ is the set of optimal arms given by $\argmin_{i\in \mathcal{A}} \textup{MV}^{\rho}_i$. With probability at least $1-K/n^3$, the regret of the RALCB algorithm at time $n$ is upper bounded by:}
\begin{equation}
   \mathcal{R}_n(\pi) \leq \frac{1}{n}\sum_{i \in \mathcal{A}\setminus \mathcal{A}^{0}} \left(\frac{8\log(n)}{\varphi^{-1}(\Delta_i/2)}+5\right)(\Delta_{i}+2\Gamma_{i, \max }^{2})+\frac{\varphi\left(\frac{8K\log (n)}{n}\right)}{\theta}+\frac{16\sqrt{2}K\theta\log (n)}{n}.
\end{equation}
\end{theorem}

\begin{remark}
\black{For $\rho=1$, the risk-averse MAB model simplifies to the classic risk-neutral MAB (i.e., a reward maximization problem). The index becomes: 
$$V^{\textup{RALCB}}_{i, t-1}=\widehat{\mu}_{i, t-1}+\theta \sqrt{\frac{8\log (t)}{T_{i,t-1}}}.$$
And, the regret upper bounds presented in Theorems \ref{Thm_1} and \ref{Thm_2} replicate the bounds on risk-neutral regret, comparable to the regret achieved by \citet{auer2002finite} who explored the UCB approach for Gaussian distributions.}\\
\black{For $\rho =0$, the risk-averse MAB model transforms into the risk-focused MAB (i.e., a variance minimization problem). The index becomes: 
$$V^{\textup{RALCB}}_{i, t-1}=\widehat{\sigma}^2_{i, t-1}-32\theta^2 \max\left(\sqrt{\frac{8\log (t)}{T_{i,t-1}}},\frac{8\log (t)}{T_{i,t-1}}\right)-\theta^2 \frac{8\log (t)}{T_{i,t-1}}.$$
The regret upper bounds, as provided in Theorems \ref{Thm_1} and \ref{Thm_2}, produce comparable results achieved by \citet{audibert2009exploration} who explored the UCB approach for bounded random variables.}
\end{remark}

\subsubsection{Drawing multiple arms}\label{section5}
In contrast to many papers on multi-armed bandits, we do not assume that the individual arms are independent. A dependence between arms arises naturally, e.g., if we allow the learner to pull multiple arms at each round. Then, the $K$-armed problem is transformed into a $P$-armed problem, where $P$ counts the total of possible arm combinations. As a result, we now have a new action set $\mathcal{A}^{*}:=\{1,...,P\}$. At time $t$, the reward $Y^{j}_{t}$ of an arm $j\in\mathcal{A}^*$ is a convex combination of the rewards $X^{i}_t$ of single arms. For example,
  $$
     \begin{bmatrix}
         Y^1_t\\
         Y^2_t\\ 
         \vdots\\ 
         Y^P_t
     \end{bmatrix}
     =
     \begin{bmatrix}
         w_{11} & w_{12} & \cdots & w_{1K}\\
         w_{21} & w_{22} & \cdots & w_{2K}\\ 
         \vdots & \vdots & \ddots & \vdots\\ 
         w_{P1} & w_{P2} & \cdots & w_{PK} 
     \end{bmatrix}
     \times
     \begin{bmatrix}
        X^1_t\\
        X^2_t\\ 
         \vdots\\ 
        X^{K}_t
     \end{bmatrix},
  $$
 where $w_{i,j} \in [0,1]$ and $\sum_{i \in \mathcal{A}}w_{i,j}=1$ for all $i\in\{1,...,K\}$ and $j \in \{1,...,P\}$. The weights $w_{i,j}$ are specified at the start and fixed over rounds. We illustrate the above transformation in the following example:
 \begin{example}\label{example}
Starting with the 3-armed bandit problem (i.e., $K=3$), we want to pull two arms at each round. The problem is then converted into a 3-armed bandit problem (i.e., $P = 3$). We assign equal weights to two arms in each combination (i.e., $w_{1,j}=w_{2,j}=0.5$ for each $j$). Then we have:
 $$
     \begin{bmatrix}
         Y^1_t\\
         Y^2_t\\ 
         Y^2_t
     \end{bmatrix}
     =
     \begin{bmatrix}
         0.5 & 0.5 & 0\\
         0.5 & 0 & 0.5\\ 
         0 & 0.5 & 0.5
     \end{bmatrix}
     \times
     \begin{bmatrix}
        X^1_t\\
        X^2_t\\ 
        X^{3}_t
     \end{bmatrix}.
  $$
 \end{example}
 Theorem 2.7 in \citet{rivasplata2012subgaussian} reveals a useful property of sub-Gaussian random variables:
\begin{proposition}
\black{Let $X$ and $Y$ be sub-Gaussian random variables with parameters $\theta_X$ and $\theta_Y$ respectively. Then for any $w \in [0,1]$, $wX+(1-w)Y$ is still a sub-Gaussian random variable with a parameter $w\theta_X+(1-w)\theta_Y$.} \end{proposition}
The above proposition states that without an independence assumption, a convex combination of sub-Gaussian random variables is still sub-Gaussian but with an updated parameter. Therefore, our results can be used when the learner can pull more than one arm during each round.

\begin{example}
    \black{To demonstrate the limited impact of drawing multiple arms on the algorithm's performance, we conducted tests on the bandit problem with $K=3$, where the arms are normally distributed and each pair of arms share the same correlation coefficients $\tau_{1,2}=\tau_{2,3}=\tau_{1,3}=0.2$. The parameter settings for each arm are $\mu =(0.1,0.41,0.79)$ and $\sigma^2=(0.05,0.44,0.85)$. Subsequently, we performed the bandit transformation based on the matrix outlined in Example \ref{example} (i.e., $P = 3$). Presented below is the cumulative regret comparison between the two cases:}

\vspace{-10mm}
\begin{figure}[H]
\centering
      \begin{overpic}[width=0.6\textwidth]{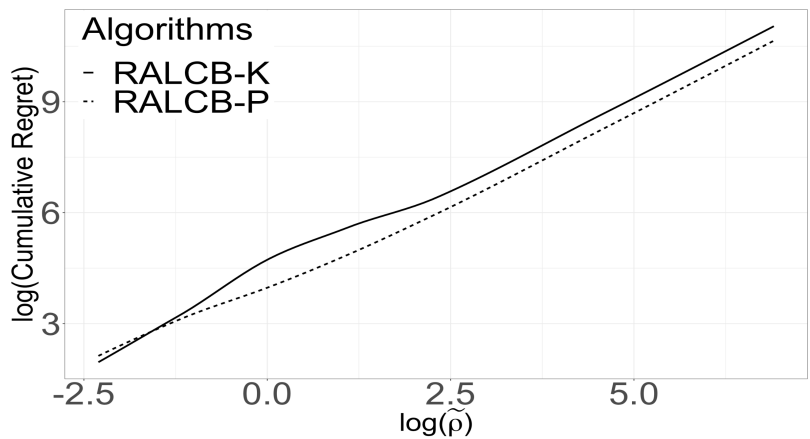}
        \put(558,128){$\widetilde{}$}
      \end{overpic}
      \vspace{-10mm}
\caption{\black{Cumulative regret comparison averaged over 1000 runs for RALCB algorithms before and after the transformation over different choices of $\widetilde\rho:=\rho/(1-\rho)$.}} \label{fig:KP}
\end{figure}
\end{example}

\black{The presented results indicate that the algorithm's performance is minimally affected by the above transformation. This observation confirms the capability of our algorithm to effectively deal with pulling multiple arms.}

\section{Numerical Experiments}\label{section6}
We perform a numerical analysis of our algorithm in this section. \citet{vakili2016risk} analyze the risk-aware MAB problem, but the square of the underlying $X$ is assumed to be sub-Gaussian in their algorithm, making comparison with our results difficult. \citet{zhu2020thompson} and \cite{liu2020risk} propose two risk-averse MAB algorithms with a commitment to good numerical performance, although both algorithms are explicitly targeting Gaussian bandits. As a result, we use the MVLCB method proposed by \citet{sani2012risk} as the benchmark for a fair comparison. Note that  \citet{sani2012risk}  use the mean-variance functional
$$\text{mv}^{\widetilde\rho}_i:=\sigma_i^2-\widetilde\rho\mu_i,
$$
whose minimization is equivalent to the minimization of 
$$\textup{MV}^{\rho}_i=(1-\rho)\sigma_i^2-\rho \mu_i=(1-\rho)\text{mv}^{\widetilde\rho}_i
$$
if we let $$\widetilde\rho:=\frac\rho{1-\rho}$$ 
as long as $\rho\in[0,1)$. To facilitate the comparison between our results and the ones in \citet{sani2012risk}, we will use $\widetilde\rho$ as the risk aversion parameter throughout this section.
A sketch of the MVLCB algorithm from \citet{sani2012risk} is reported in Algorithm \ref{algo2}. 

\begin{algorithm}[H]
\caption{The Mean-Variance Lower Confidence Bound Algorithm (MVLCB)}\label{algo2}
\footnotesize
\begin{algorithmic}[1]
\State \textbf{Input:} \black{$n$, $\theta$, $K$, $\widetilde\rho$}.
\State \textbf{for each} $t=1,2,...,K$ \textbf{do} 
\State \qquad Play arm $\pi_t=t$ and observe $X^{\pi_t}_1$.
\State \textbf{end for}
\State Update $\widehat{\textup{MV}}^{\widetilde\rho}_{i,1}$ for $i=1,...,K$ by Equation (\ref{Eq_3}).
\State Set $T_{i,K}=1$ for $i=1,...,K$.

\State \textbf{for each} \black{$t=K+1,K+2,...,n$} \textbf{do} 
\State \qquad \textbf{for each} $i=1,2,...,K$ \textbf{do} 

\State \qquad \qquad Compute $V^{\textup{MVLCB}}_{i, t-1}=\widehat{\mathrm{MV}}^{\widetilde\rho}_{i, t-1}-(5+\widetilde\rho)\sqrt{\frac{8 \log (n)}{\black{T_{i,t-1}}}}.$
\State \qquad \textbf{end for}
\State \qquad Return $\pi_{t}=\argmin _{i=1,...,K} V^{\textup{MVLCB}}_{i, T_{i,t-1}}.$
\State \qquad Update $T_{i,t}=T_{i,t-1}+1$ for $i=1,...,K$.
\State \qquad Observe $X^{\pi_t}_{T_{i,t}}.$
\State \qquad Update $\widehat{\textup{MV}}^{\widetilde\rho}_{i,T_{i,t}}$ by Equation (\ref{Eq_3}).
\State \textbf{end for}
\end{algorithmic}
\end{algorithm}

We report numerical simulations to validate our theoretical results in the previous sections. Here, we consider the bandit problem with $K=15$ arms, where the reward of each arm is independent and follows a Gaussian distribution. The parameter setting for each arm is the same as the experiments from \citet{sani2012risk}, which guarantees that the observations lie in the interval $[0,1]$ with $95\%$ confidence.
\begin{itemize}
    \item $\mu =(0.1, 0.2, 0.23, 0.27, 0.32, 0.32, 0.34, 0.41, 0.43, 0.54, 0.55, 0.56, 0.67, 0.71, 0.79).$
    \item $\sigma^2=(0.05, 0.34, 0.28, 0.09, 0.23,
0.72, 0.19, 0.14, 0.44, 0.53, 0.24, 0.36, 0.56, 0.49, 0.85).$
\end{itemize}
We compare the performance of two algorithms under three scenarios: 
\begin{itemize}
    \item Variance minimization: $\widetilde\rho=10^{-3}$, optimal arm is $i^{0.001}_0=1.$
    \item Risk-return balance: $\widetilde\rho=1$, optimal arm is $i^{1}_0=11.$
    \item Reward maximization: $\widetilde\rho=10^{3}$, optimal arm is $i^{1000}_0=15.$
\end{itemize}
\subsection{Variance Minimization Scenario}
In the variance minimization scenario (i.e., $\widetilde\rho=10^{-3}$), the optimal arm is $i^{0.001}_0=1$. We run both algorithms over 1000 runs with a time horizon of $n=30000$. In Figure \ref{fig:1}, we record the pulls of arms at each time point for two algorithms.

\begin{figure}[H]
    \subfloat[Arm pulling record by the MVLCB algorithm]{%
      \begin{overpic}[width=0.48\textwidth]{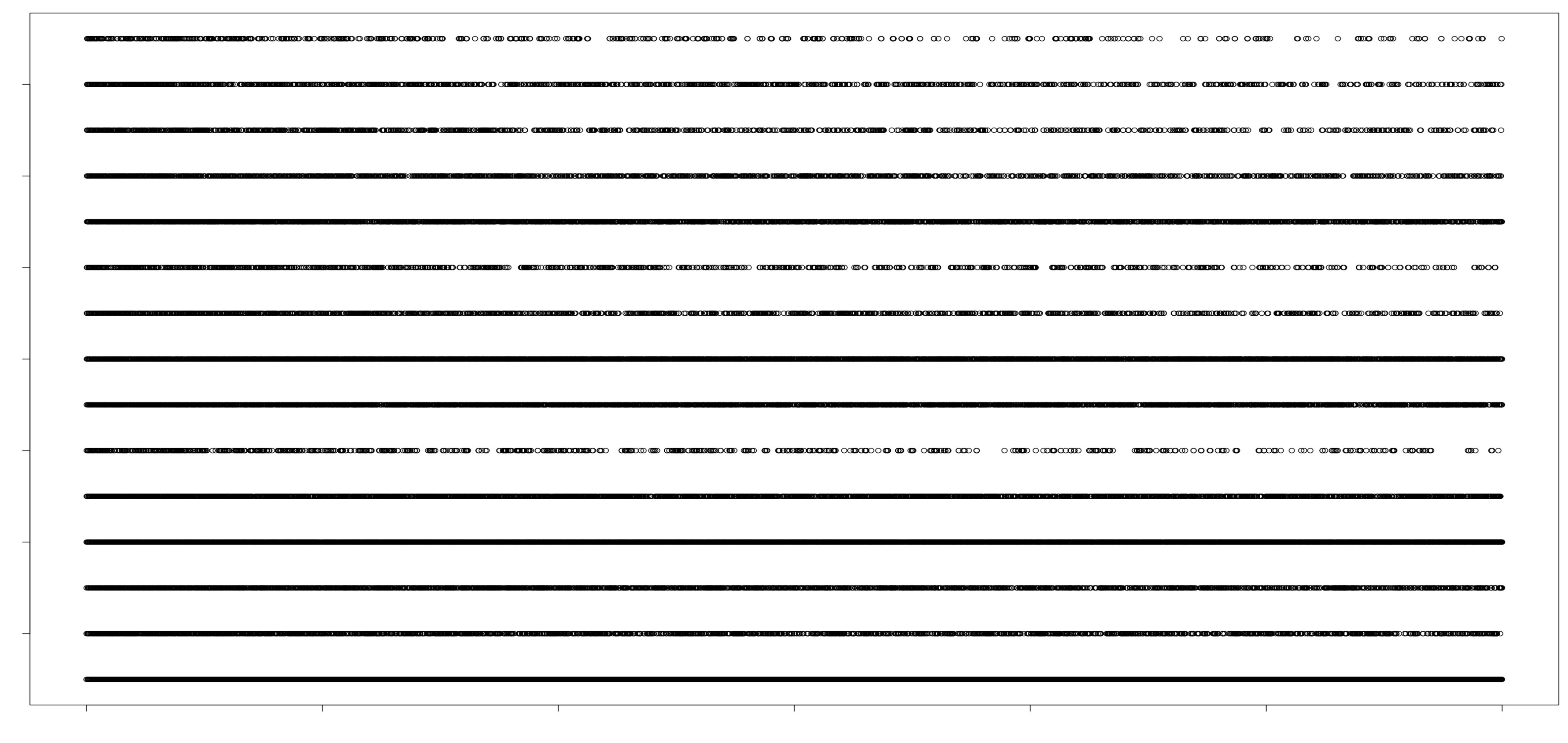}
        \put(-10,60){\scriptsize 2}
        \put(-10,120){\scriptsize 4}
        \put(-10,180){\scriptsize6}
        \put(-10,240){\scriptsize8}
        \put(-20,300){\scriptsize10}
        \put(-20,360){\scriptsize12}
        \put(-20,420){\scriptsize14}
        \put(60,-10){\scriptsize0}
        \put(170,-10){\scriptsize5000}
        \put(300,-10){\scriptsize10000}
        \put(450,-10){\scriptsize15000}
        \put(600,-10){\scriptsize20000}
        \put(750,-10){\scriptsize25000}
        \put(900,-10){\scriptsize30000}
      \end{overpic}
    }\hfill
    \subfloat[Arm pulling record by the RALCB algorithm]{%
      \begin{overpic}[width=0.48\textwidth]{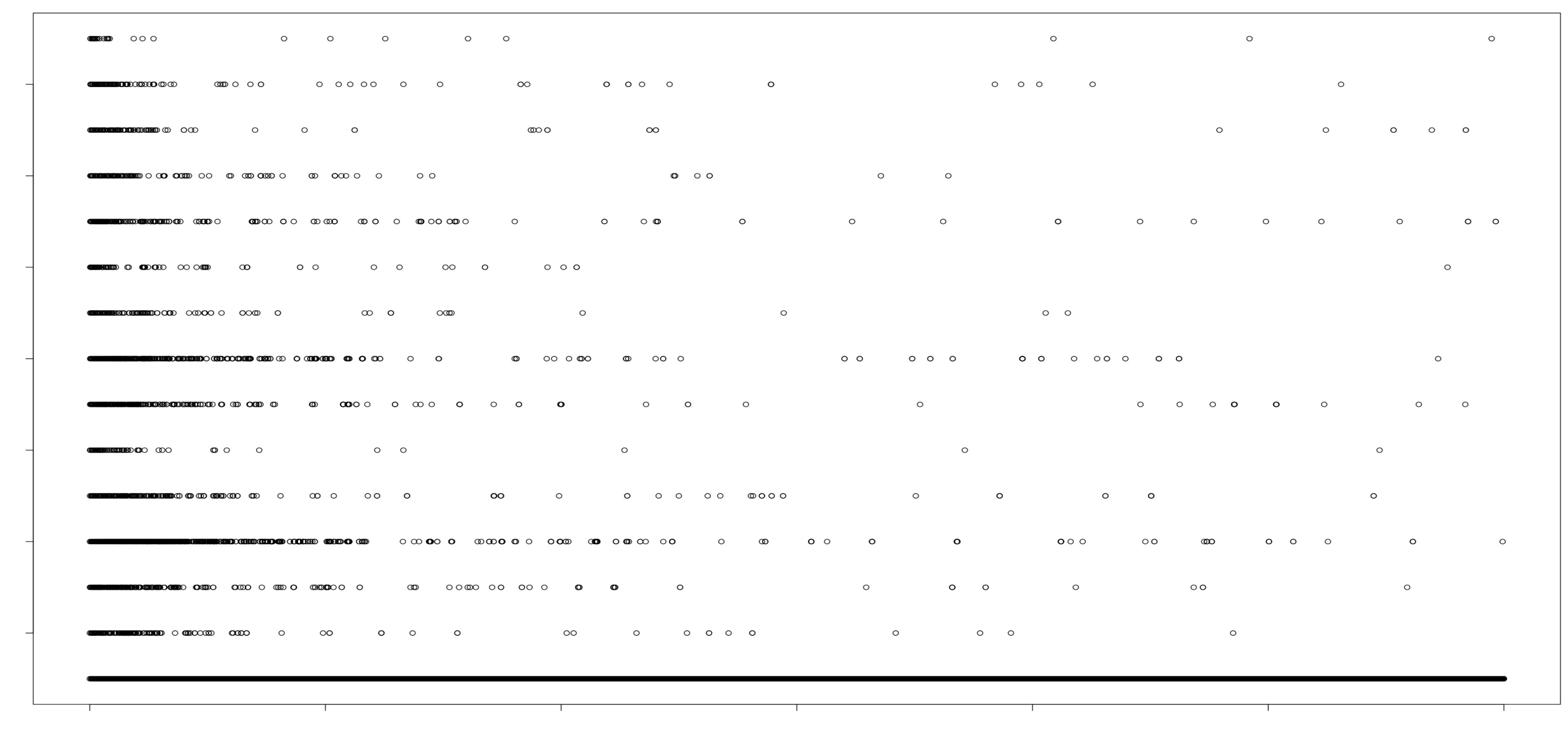}
         \put(-10,60){\scriptsize 2}
        \put(-10,120){\scriptsize 4}
        \put(-10,180){\scriptsize6}
        \put(-10,240){\scriptsize8}
        \put(-20,300){\scriptsize10}
        \put(-20,360){\scriptsize12}
        \put(-20,420){\scriptsize14}
        \put(60,-10){\scriptsize0}
        \put(170,-10){\scriptsize5000}
        \put(300,-10){\scriptsize10000}
        \put(450,-10){\scriptsize15000}
        \put(600,-10){\scriptsize20000}
        \put(750,-10){\scriptsize25000}
        \put(900,-10){\scriptsize30000}
      \end{overpic}
    }
    \caption{Record of arms pulled at each time point in a single simulation for two algorithms under $\widetilde\rho=10^{-3}$.  \label{fig:1}}
\end{figure}
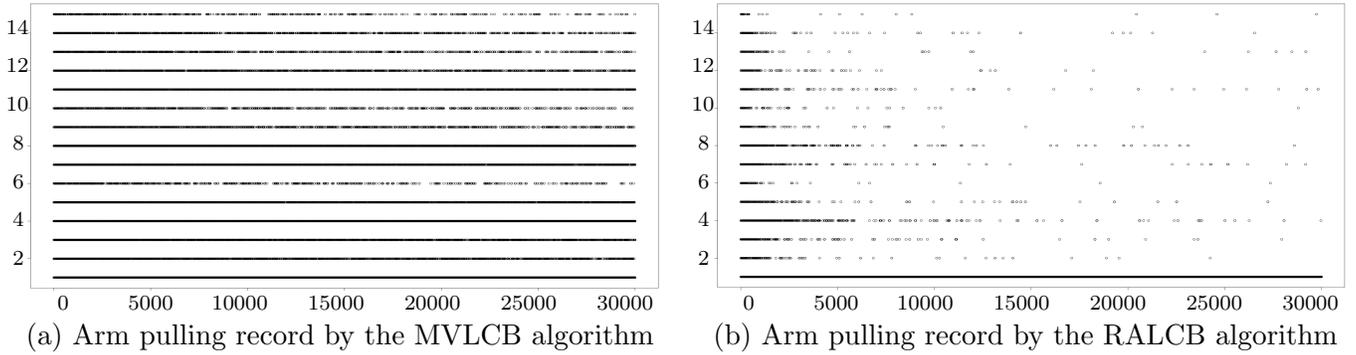

From Figure \ref{fig:1}, we can conclude that:

\begin{enumerate}
 \item Under the RALCB algorithm, the optimal arm (i.e., $i^{0.001}_0=1$) has been pulled more frequently than under the MVLCB algorithm.

\item In the long run, the RALCB algorithm starts picking the best arms consistently, with a small chance of exploration.

\item Under the variance minimization scenario, the MVLCB algorithm fails to identify the optimal arms.

\end{enumerate}
In Figure \ref{fig:2}, we present the cumulative regret and mean regret, which is averaged over 1000 runs with a time horizon of $n=30000$.

\begin{figure}[H]
    \subfloat[Cumulative regret]{%
      \begin{overpic}[width=0.48\textwidth]{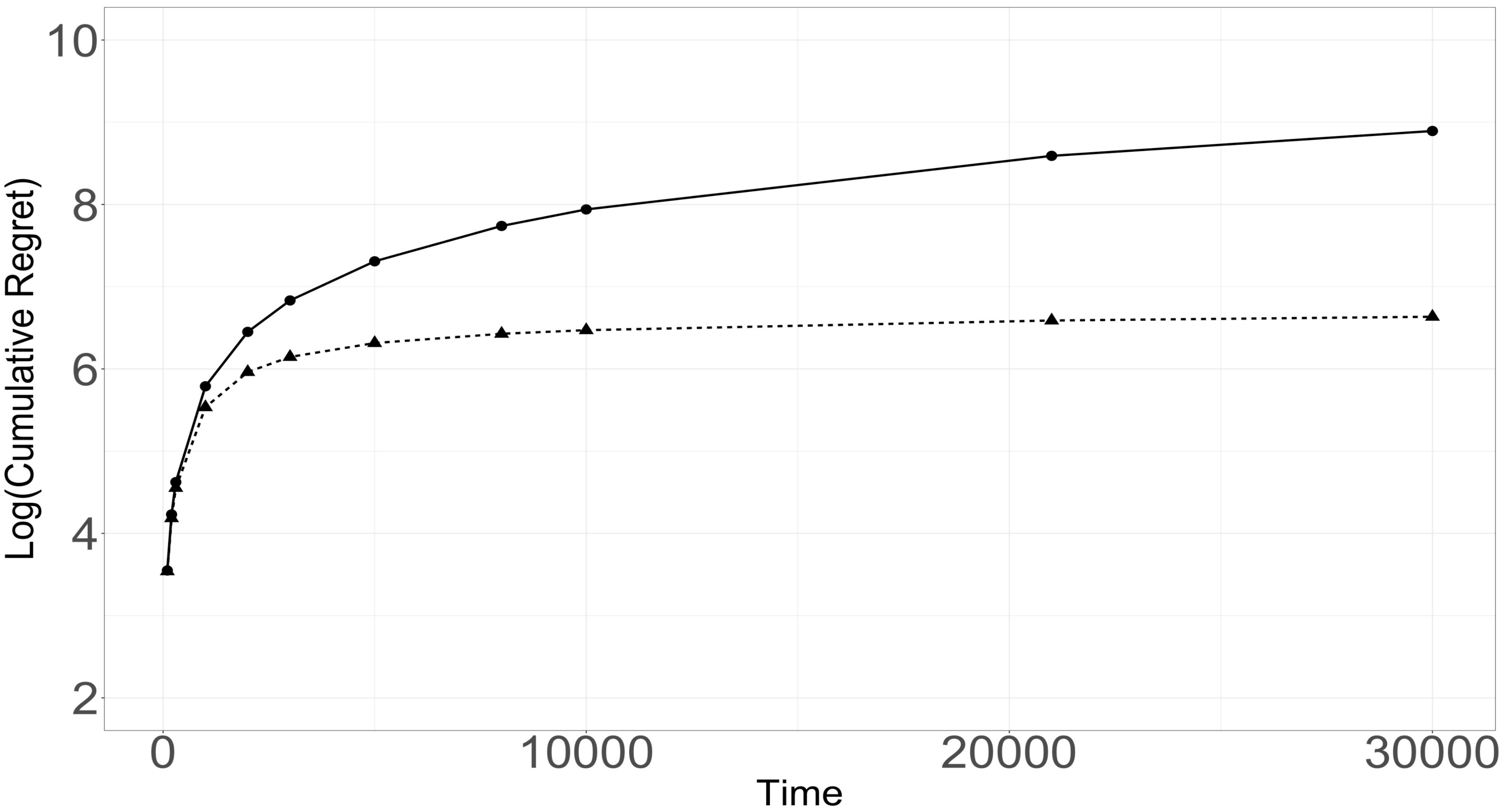}
        \put(300,450){MVLCB}
        \put(300,250){RALCB}
      \end{overpic}
    }\hfill
    \subfloat[Mean regret]{%
      \begin{overpic}[width=0.48\textwidth]{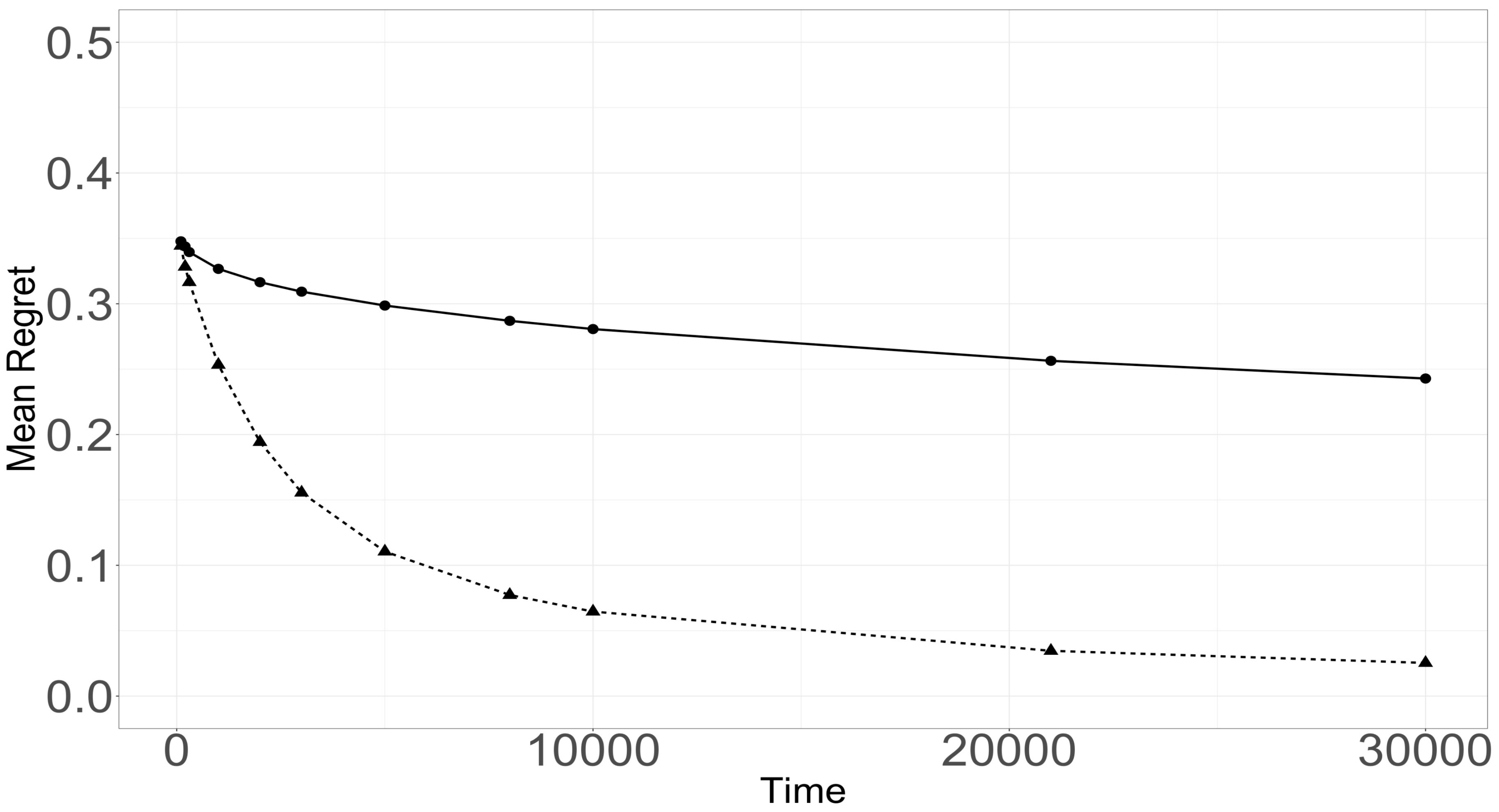}
       \put(300,400){MVLCB}
        \put(300,200){RALCB}
      \end{overpic}
    }
    \caption{Regret performance averaged over 1000 runs for two algorithms under $\widetilde\rho=10^{-3}$. \label{fig:2}}
\end{figure}

From Figure \ref{fig:2}, we can conclude that:
\begin{enumerate}
    \item The RALCB algorithm achieves a lower cumulative regret than the MVLCB algorithm. The cumulative regret of the RALCB algorithm goes flat faster than that of the MVLCB algorithm.
    \item The RALCB algorithm outperforms the MVLCB algorithm in terms of mean regret. The decreasing rate of the mean regret for RALCB is higher than the one for MVLCB.
\end{enumerate}

\subsection{Risk-return Balance Scenario}
In the risk-return balance scenario (i.e., $\widetilde\rho=1$), the optimal arm is $i^{1}_0=11$. We run both algorithms over 1000 runs with a time horizon of $n=30000$. In Figure \ref{fig:3}, we record the pulls of arms at each time point for two algorithms.

\begin{figure}[H]
    \subfloat[Arm pulling record by the MVLCB algorithm]{%
      \begin{overpic}[width=0.48\textwidth]{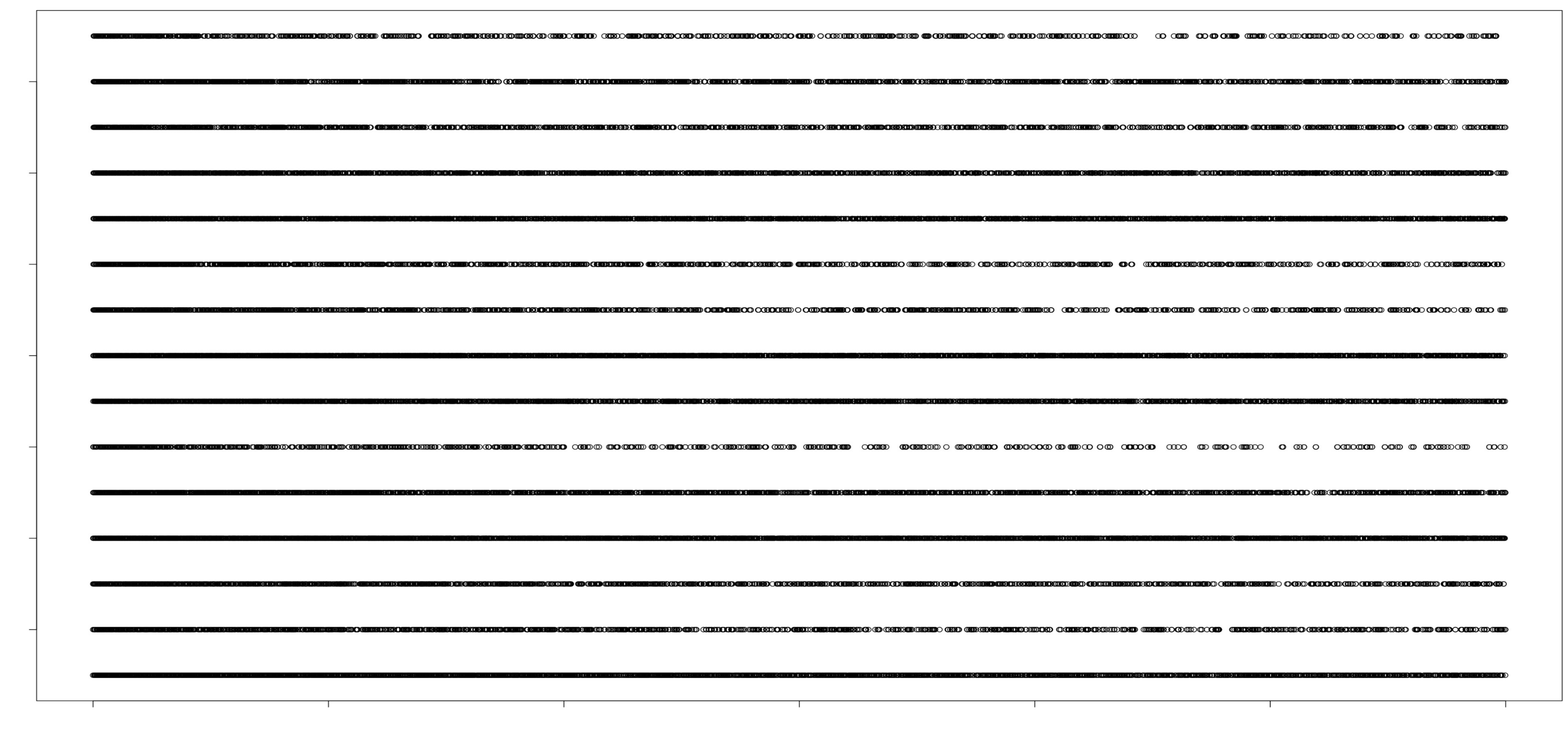}
            \put(-10,60){\scriptsize 2}
        \put(-10,120){\scriptsize 4}
        \put(-10,180){\scriptsize6}
        \put(-10,240){\scriptsize8}
        \put(-20,300){\scriptsize10}
        \put(-20,360){\scriptsize12}
        \put(-20,420){\scriptsize14}
        \put(60,-10){\scriptsize0}
        \put(170,-10){\scriptsize5000}
        \put(300,-10){\scriptsize10000}
        \put(450,-10){\scriptsize15000}
        \put(600,-10){\scriptsize20000}
        \put(750,-10){\scriptsize25000}
        \put(900,-10){\scriptsize30000}
      \end{overpic}
    }\hfill
    \subfloat[Arm pulling record by the RALCB algorithm]{%
      \begin{overpic}[width=0.48\textwidth]{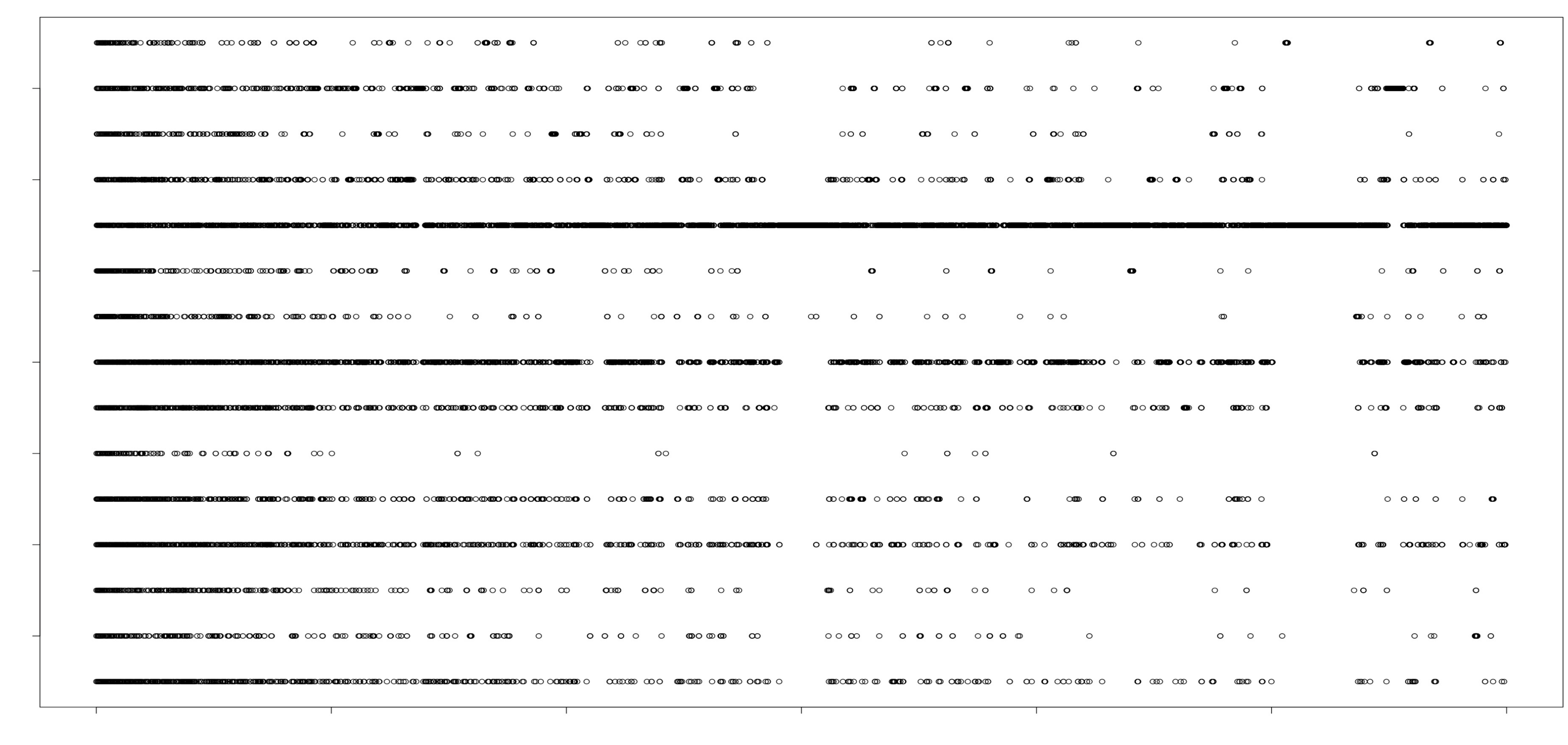}
             \put(-10,60){\scriptsize 2}
        \put(-10,120){\scriptsize 4}
        \put(-10,180){\scriptsize6}
        \put(-10,240){\scriptsize8}
        \put(-20,300){\scriptsize10}
        \put(-20,360){\scriptsize12}
        \put(-20,420){\scriptsize14}
        \put(60,-10){\scriptsize0}
        \put(170,-10){\scriptsize5000}
        \put(300,-10){\scriptsize10000}
        \put(450,-10){\scriptsize15000}
        \put(600,-10){\scriptsize20000}
        \put(750,-10){\scriptsize25000}
        \put(900,-10){\scriptsize30000}

      \end{overpic}
    }
    \caption{Record of arms pulled at each time point in a single simulation  for two algorithms under $\widetilde\rho=1$. \label{fig:3}}
\end{figure}
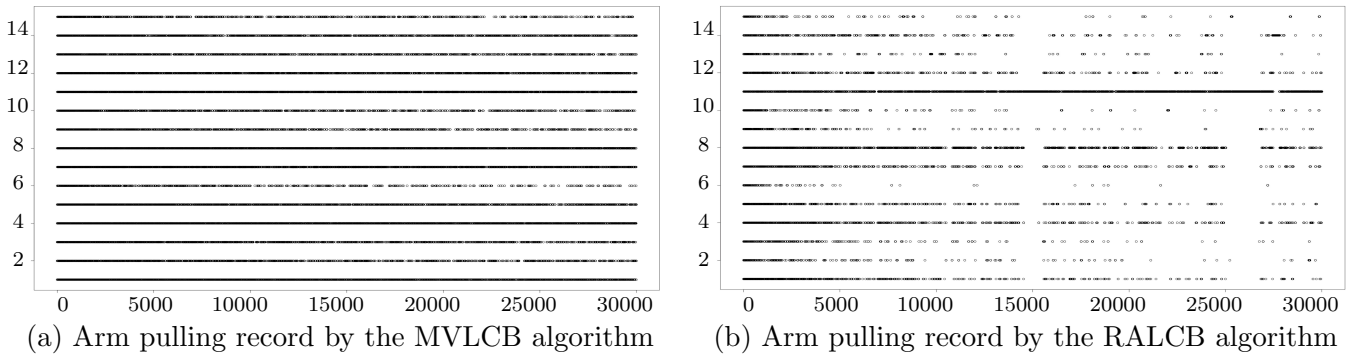

We can deduce from Figure \ref{fig:3} that:

\begin{enumerate}

\item The optimal arm (i.e., $i^{1}_0=11$) has been pulled more frequently under the RALCB algorithm than under the MVLCB algorithm.

\item In the long run, the RALCB algorithm can find the best arms with a low chance of picking suboptimal arms with the same mean-variance.

\item Under the risk-return balance scenario, the MVLCB algorithm fails to identify the optimal arms.

\end{enumerate}

In Figure \ref{fig:4}, we present the cumulative regret and mean regret, which is averaged over 1000 runs with a time horizon of $n=30000$.
\begin{figure}[H]
    \subfloat[Cumulative regret]{%
      \begin{overpic}[width=0.5\textwidth]{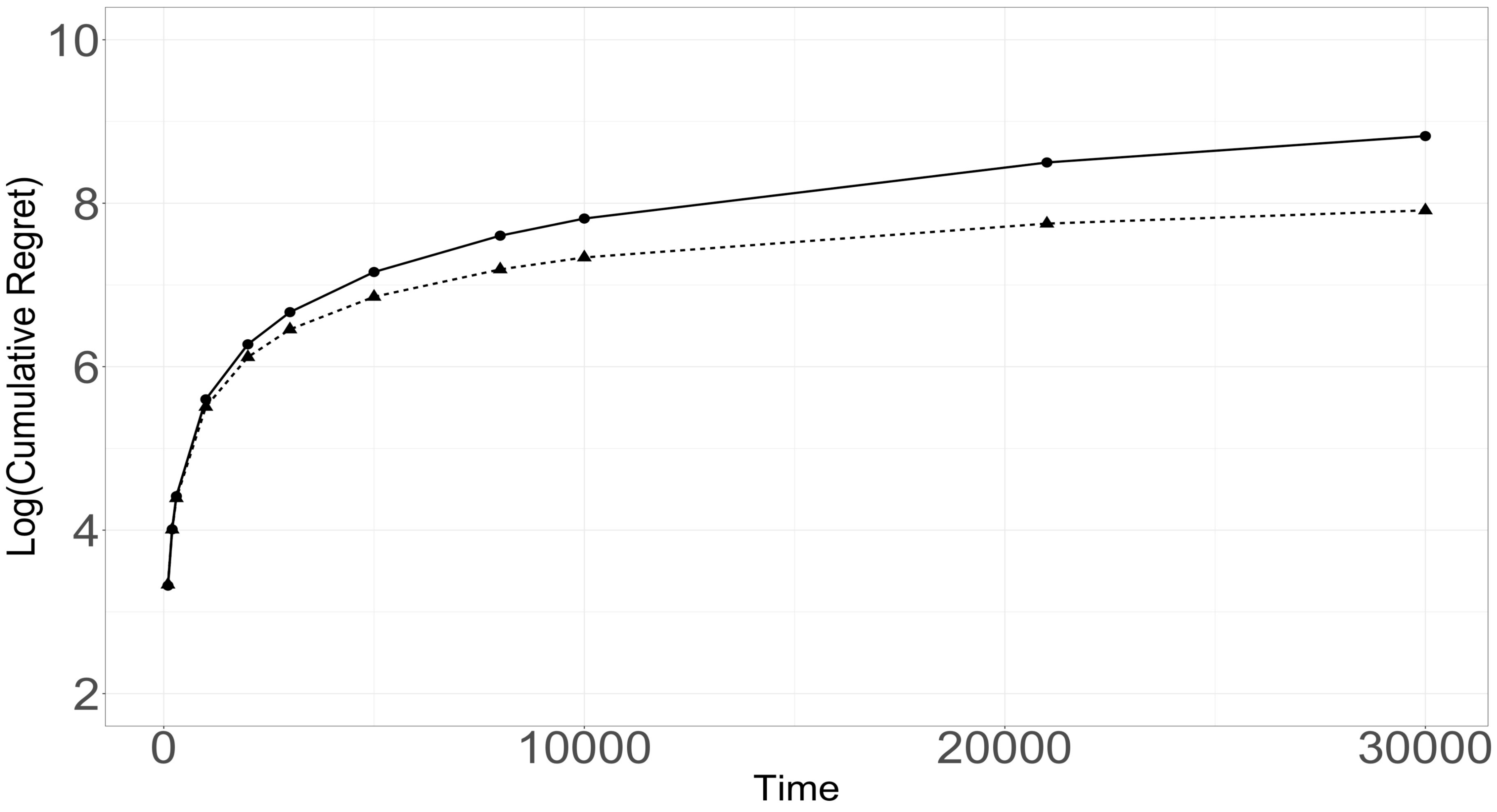}
        \put(300,450){MVLCB}
        \put(300,250){RALCB}
      \end{overpic}
    }\hfill
    \subfloat[Mean regret]{%
      \begin{overpic}[width=0.5\textwidth]{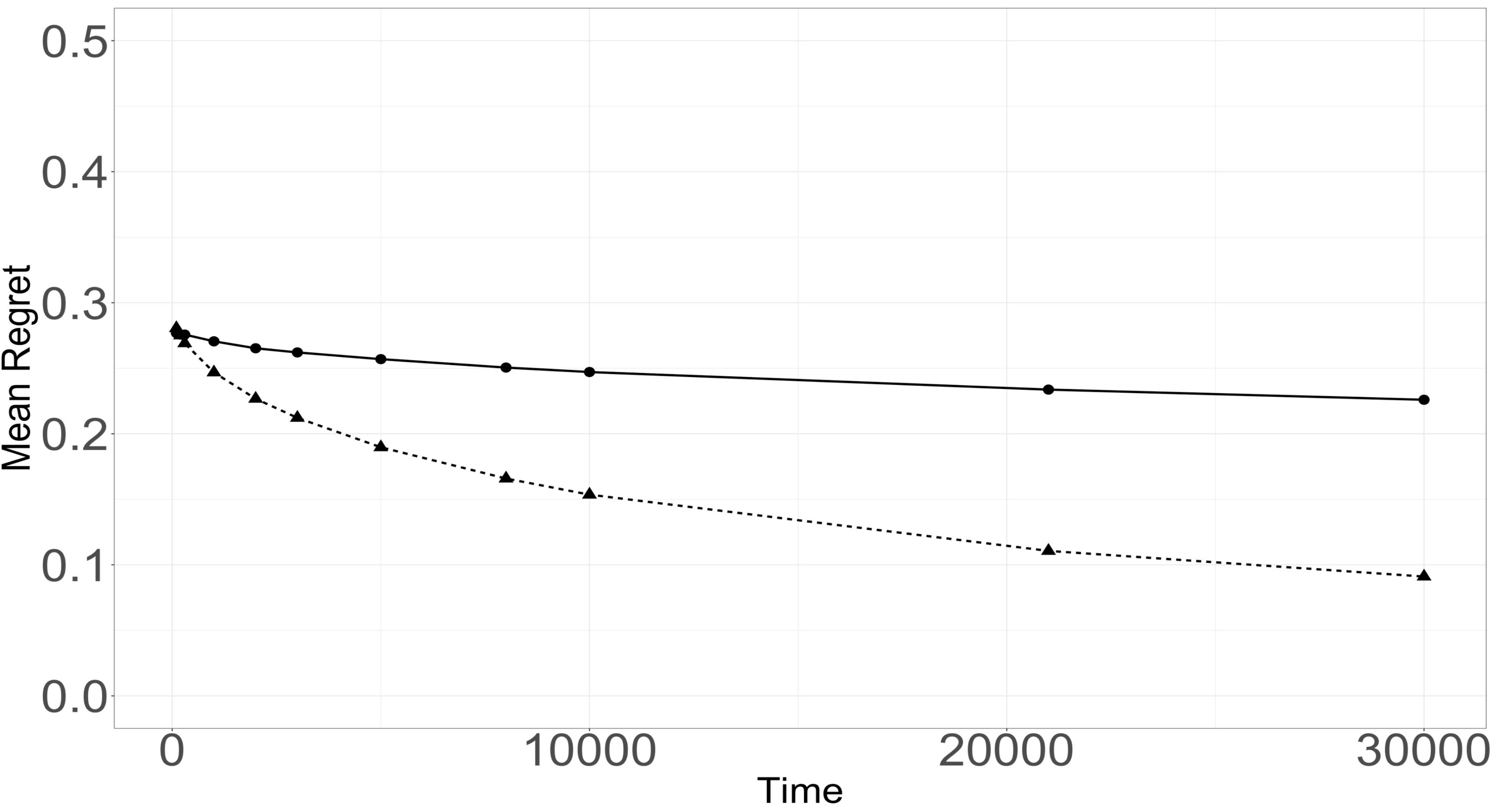}
       \put(300,350){MVLCB}
        \put(300,150){RALCB}
      \end{overpic}
    }
    \caption{Regret performance averaged over 1000 runs for two algorithms under $\widetilde\rho=1$. \label{fig:4}}
\end{figure}

From Figure \ref{fig:4}, we can conclude that:
\begin{enumerate}
    \item The RALCB algorithm achieves a lower cumulative regret than the MVLCB algorithm.
    \item The RALCB algorithm outperforms the MVLCB algorithm in terms of mean regret. The decreasing rate of the mean regret for RALCB is higher than the one for MVCLB.
\end{enumerate}
\subsection{Reward Maximization Scenario}
In the reward maximization scenario (i.e., $\widetilde\rho=10^3$), the optimal arm is $i^{1000}_0=15$. We run both algorithms over 1000 runs with a time horizon of $n=30000$. In Figure \ref{fig:5}, we record the pulls of arms at each time point for two algorithms.

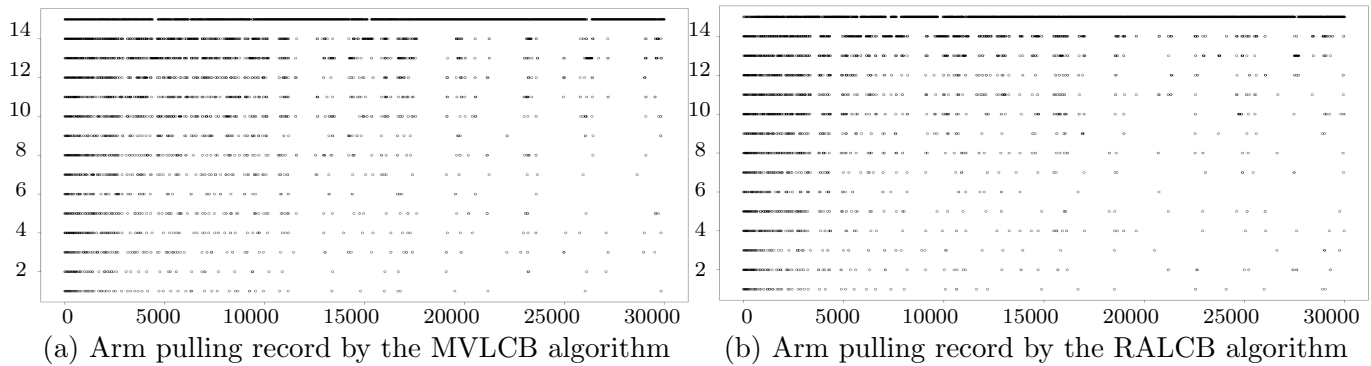
\begin{figure}[H]
    \subfloat[Arm pulling record by the MVLCB algorithm]{%
      \begin{overpic}[width=0.5\textwidth]{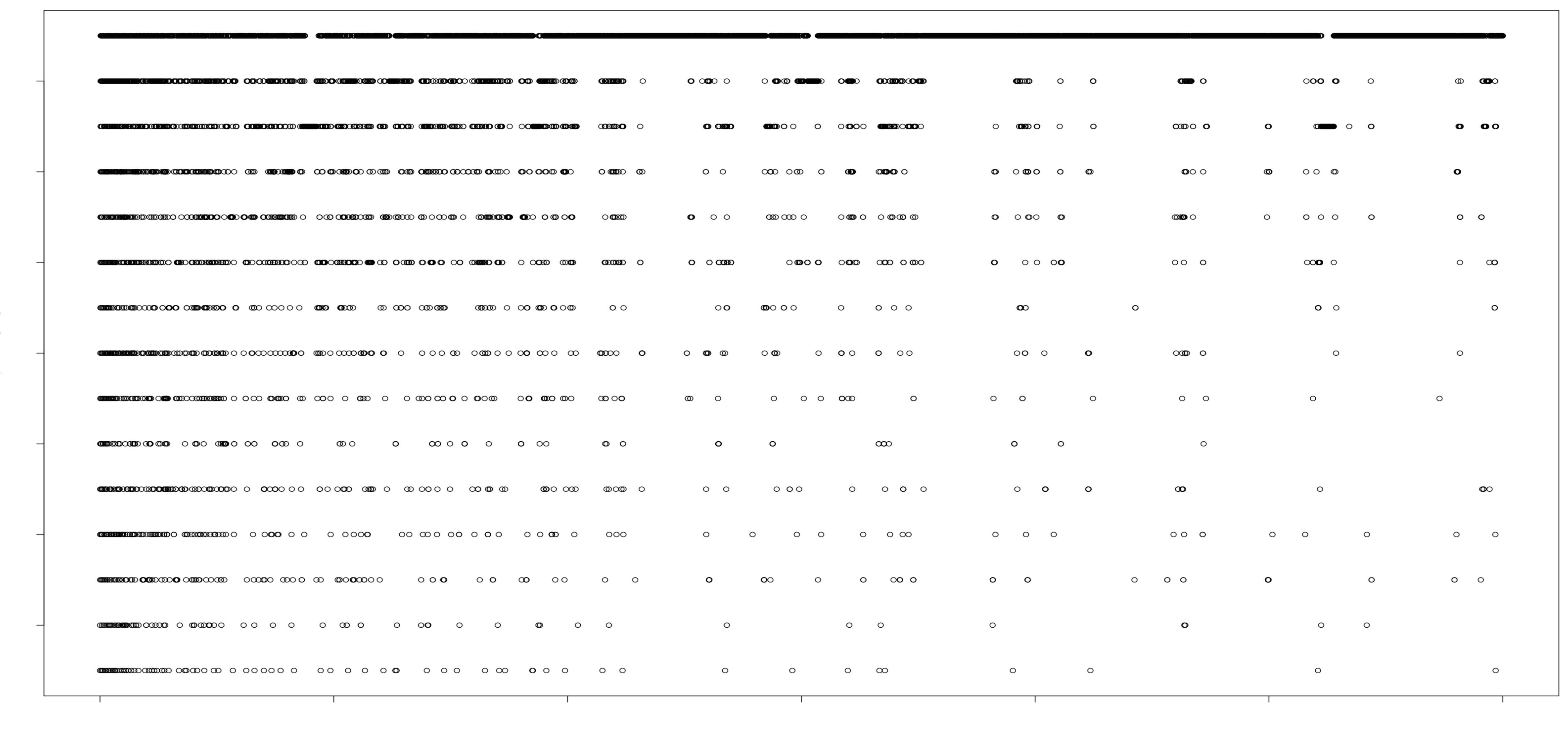}
           \put(-10,60){\scriptsize 2}
        \put(-10,120){\scriptsize 4}
        \put(-10,180){\scriptsize6}
        \put(-10,240){\scriptsize8}
        \put(-20,300){\scriptsize10}
        \put(-20,360){\scriptsize12}
        \put(-20,420){\scriptsize14}
        \put(60,-10){\scriptsize0}
        \put(170,-10){\scriptsize5000}
        \put(300,-10){\scriptsize10000}
        \put(450,-10){\scriptsize15000}
        \put(600,-10){\scriptsize20000}
        \put(750,-10){\scriptsize25000}
        \put(900,-10){\scriptsize30000}
       
      \end{overpic}
    }\hfill
    \subfloat[Arm pulling record by the RALCB algorithm]{%
      \begin{overpic}[width=0.5\textwidth]{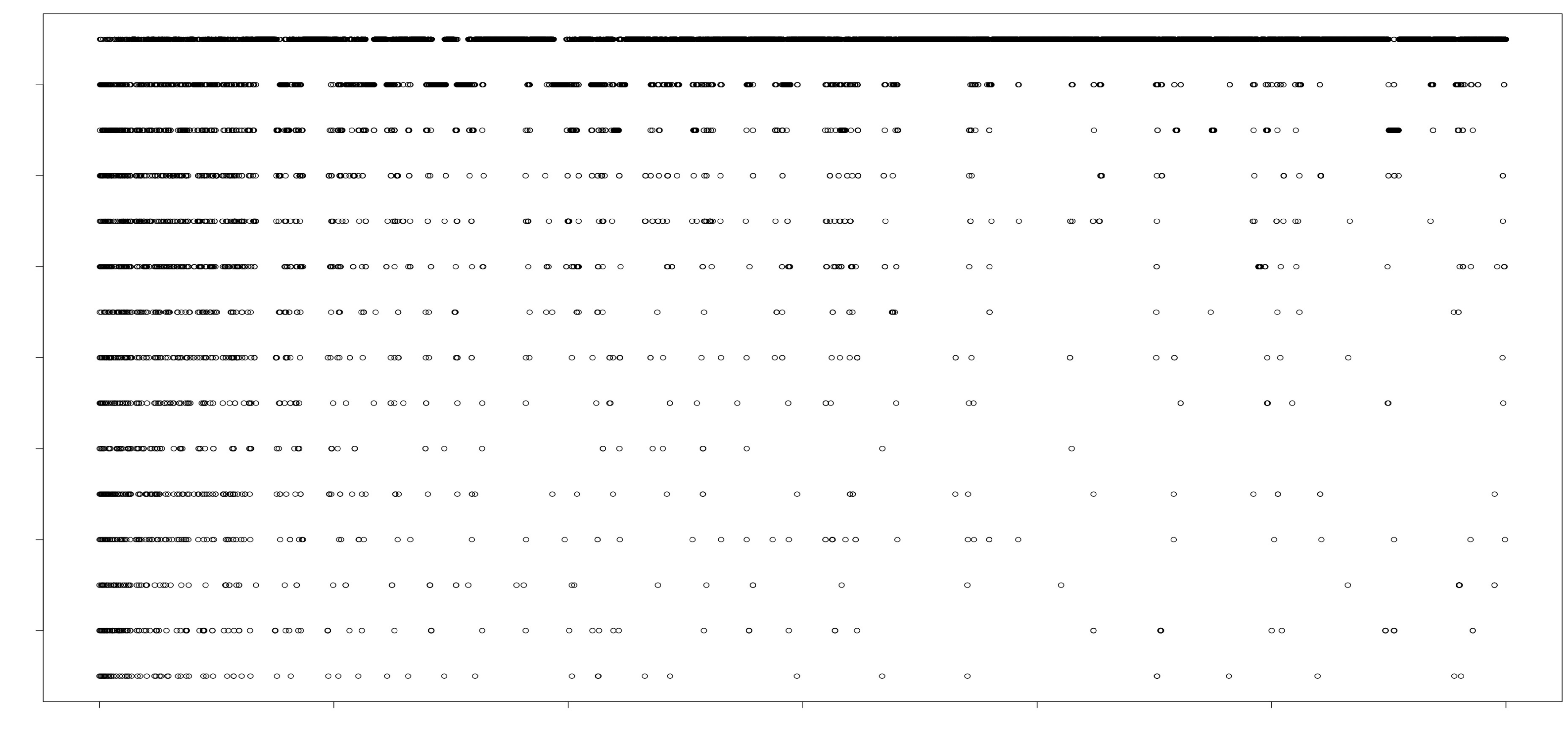}
             \put(-10,60){\scriptsize 2}
        \put(-10,120){\scriptsize 4}
        \put(-10,180){\scriptsize6}
        \put(-10,240){\scriptsize8}
        \put(-20,300){\scriptsize10}
        \put(-20,360){\scriptsize12}
        \put(-20,420){\scriptsize14}
        \put(60,-10){\scriptsize0}
        \put(170,-10){\scriptsize5000}
        \put(300,-10){\scriptsize10000}
        \put(450,-10){\scriptsize15000}
        \put(600,-10){\scriptsize20000}
        \put(750,-10){\scriptsize25000}
        \put(900,-10){\scriptsize30000}
        
      \end{overpic}
    }
    \caption{Record of arms pulled at each time point in a single simulation for two algorithms under $\widetilde\rho=10^3$. \label{fig:5}}
\end{figure}

From Figure \ref{fig:5}, we can conclude that:
\begin{enumerate}
    \item In the long run, both algorithms can identify the optimal arms with a small probability of exploration.
\end{enumerate}
In Figure \ref{fig:6}, we present the cumulative regret and mean regret, which is averaged over 1000 runs with a time horizon of $n=30000$. 
\begin{figure}[H]
    \subfloat[Cumulative regret]{%
      \begin{overpic}[width=0.5\textwidth]{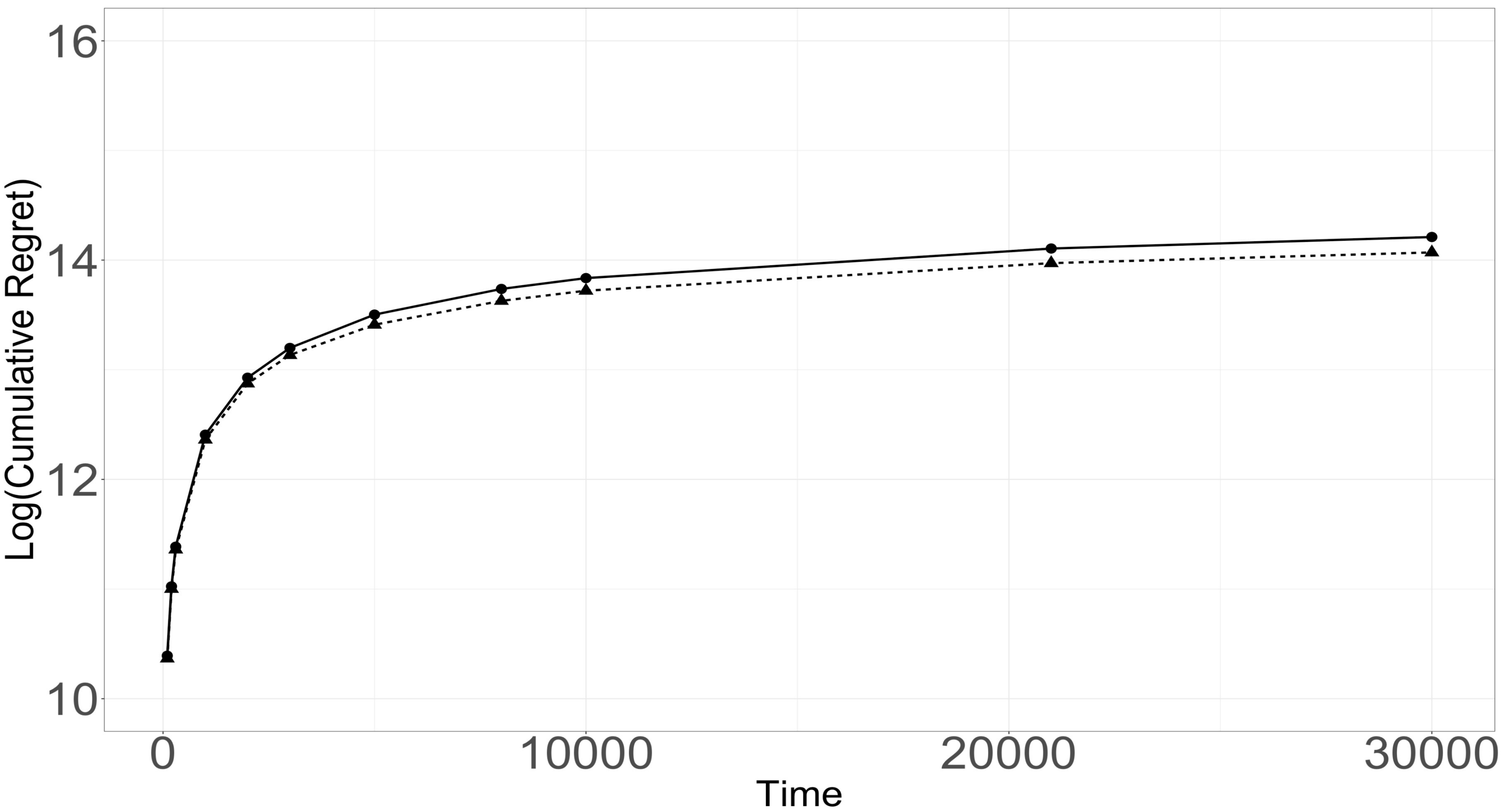}
        \put(300,400){MVLCB}
        \put(300,250){RALCB}
      \end{overpic}
    }\hfill
    \subfloat[Mean regret]{%
      \begin{overpic}[width=0.5\textwidth]{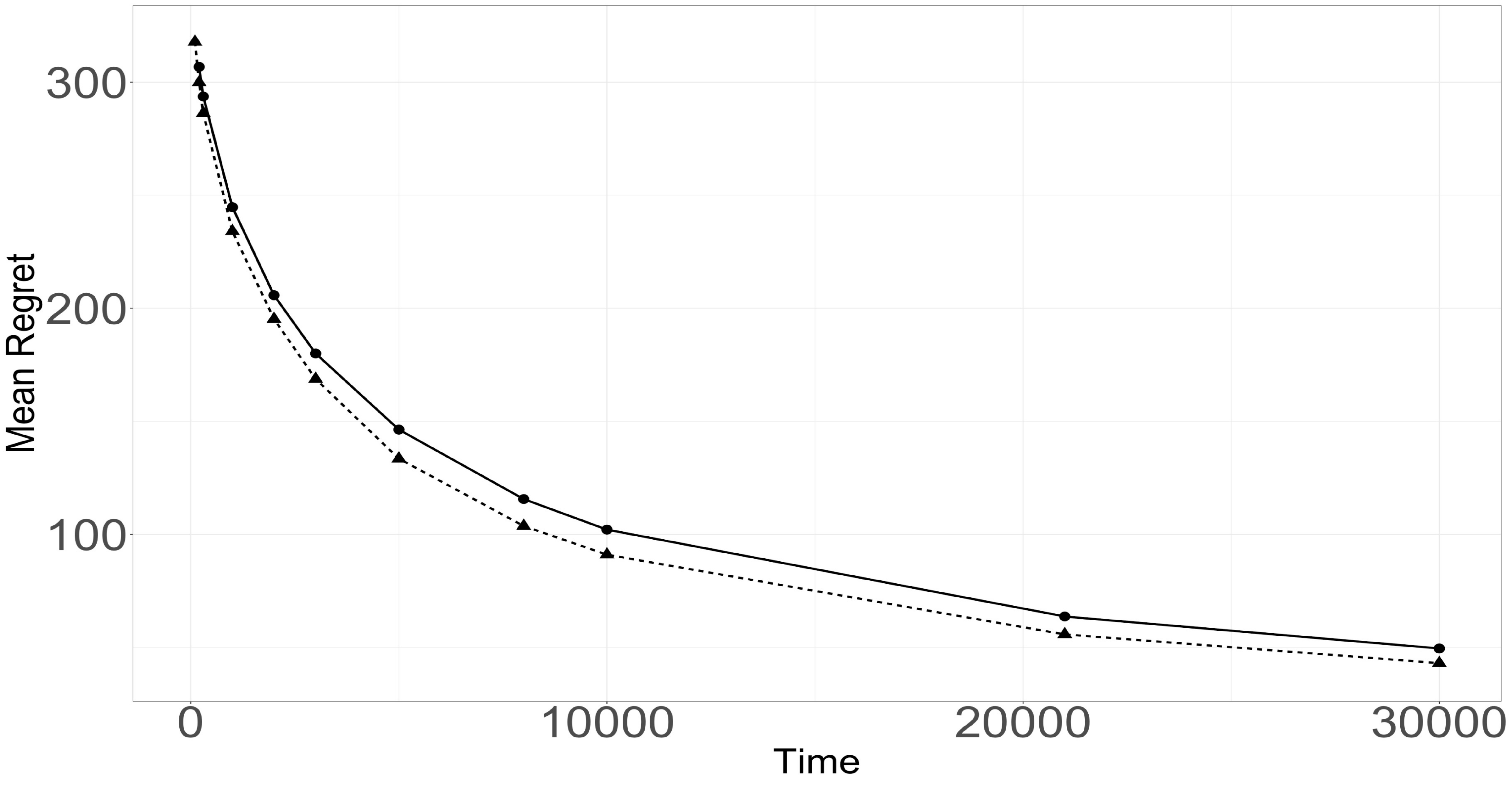}
       \put(300,250){MVLCB}
        \put(300,100){RALCB}
      \end{overpic}
    }
    \caption{Regret performance averaged over 1000 runs for two algorithms under $\widetilde\rho=10^3$. \label{fig:6}}
\end{figure}

From Figure \ref{fig:6}, we can conclude that:
\begin{enumerate}
    \item Two algorithms perform similarly, with a high cumulative regret.
    \item The RALCB algorithm slightly outperforms the MVLCB algorithm in terms of mean regret. The decreasing rate of the mean regret for RALCB is similar to the one for MVCLB.
\end{enumerate}

\subsection{\black{Algorithm Evaluation and Comparison}}
In order to validate our algorithms further and to observe how they perform in terms of $\widetilde\rho$, we ran our algorithms with different choices of $\widetilde\rho$:
$$\widetilde\rho=(0,0.001,0.01,0.1,0.3,1,3,5,7,10,20,50,100,1000,10000).$$
The time horizon $n=10000$ is fixed, and the regret is averaged over 1000 runs. We record the total number of pulls for each arm by time $n$ for two algorithms in Figure \ref{fig:7}.

\begin{figure}[H]
  \centering
  \includegraphics[width=0.5\linewidth]{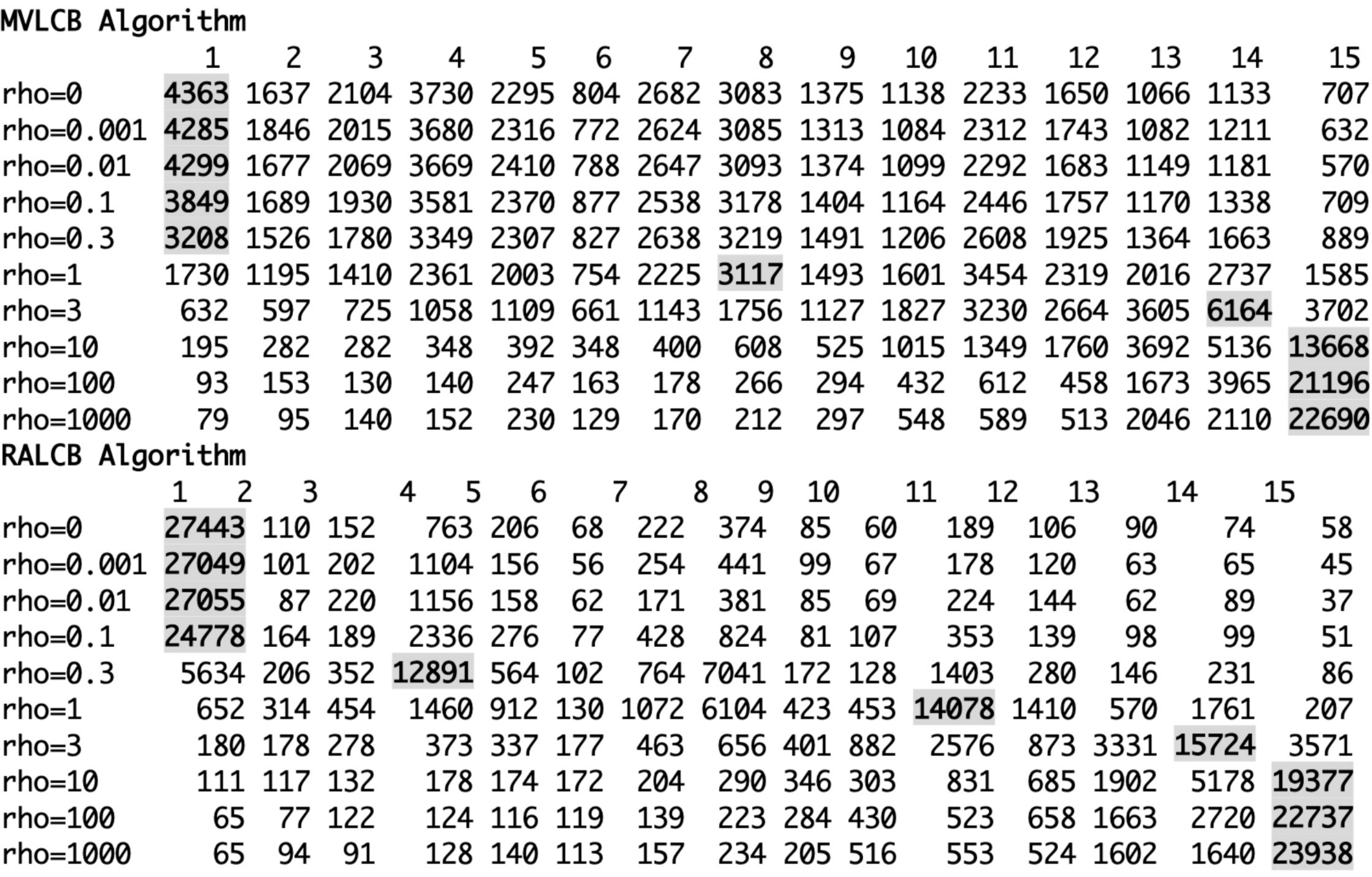}
\caption{Record of total arms pulled in a single simulation by two algorithms under different $\widetilde\rho$. The highlighted ones are the optimal arms under each $\widetilde\rho$.\label{fig:7}}
\end{figure}
From Figure \ref{fig:7}, we can conclude that:
\begin{enumerate}
    \item Under all different $\widetilde\rho$, the RALCB algorithm pulls the optimal arms more frequently than the MVLCB algorithm.
    \item Both algorithms can distinguish the optimal arms from suboptimal arms better when $\widetilde\rho$ is small. 
\end{enumerate}
We report the cumulative regrets in Figure \ref{fig:8}, 
\begin{figure}[H]
\centering
      \begin{overpic}[width=0.6\textwidth]{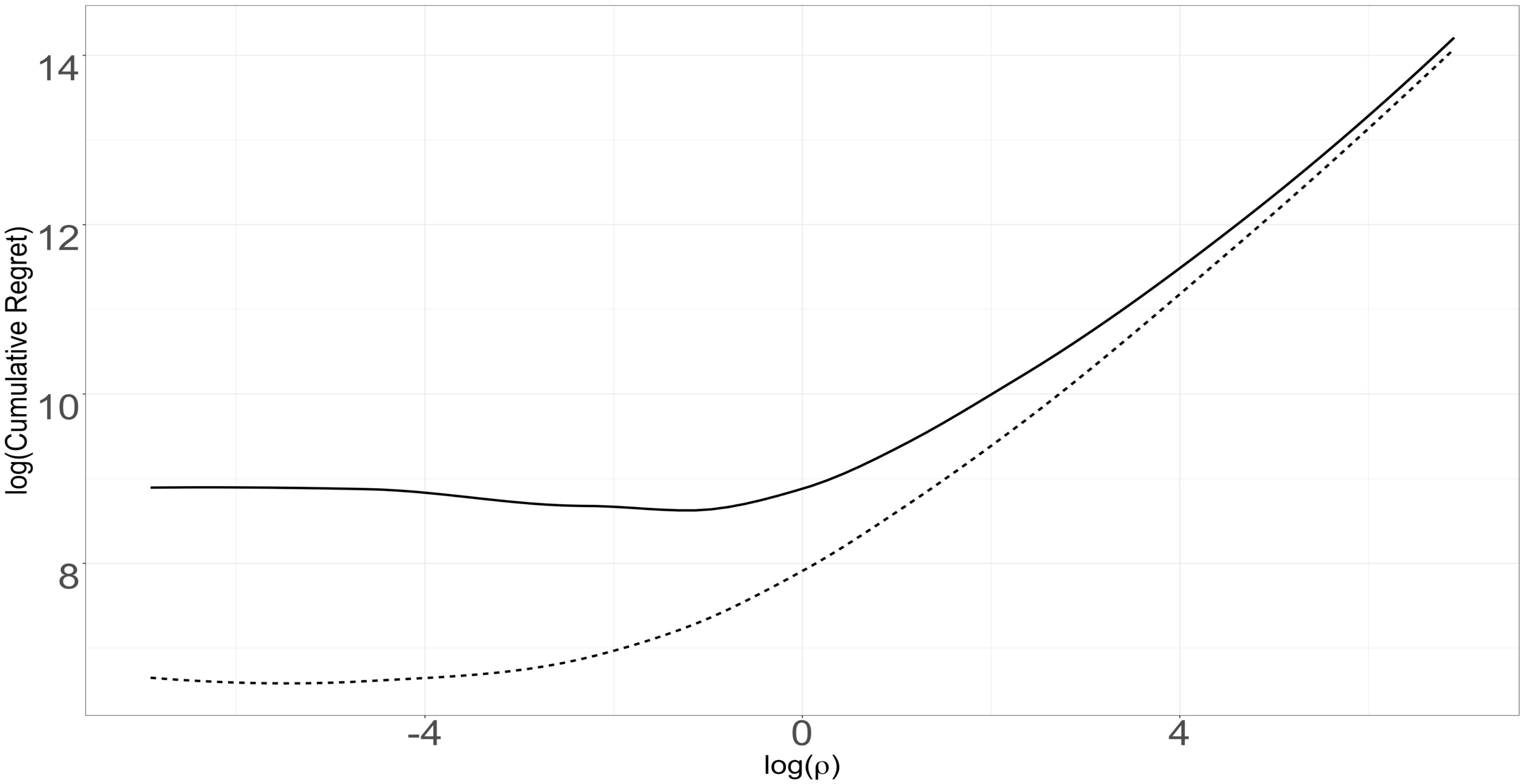}
        \put(400,200){MVLCB}
        \put(600,100){RALCB}
        \put(538,5){\scriptsize{$\widetilde{}$}}
      \end{overpic}
    \caption{Cumulative regret comparison averaged over 1000 runs for two algorithms over different choices of $\widetilde\rho$. \label{fig:8}}
\end{figure}
From Figure \ref{fig:8}, we observe the outperformance of RALCB over MVLCB when $\widetilde\rho$ is small. But, the performance gap narrows as $\widetilde\rho$ increases.

\black{It is evident from Theorem \ref{Thm_2} that when the number of arms, $K$, is large, the regret bound increases. To empirically examine this correlation between average regret and the growing number of arms, we conducted a numerical test. Every arm in the simulation is considered to be an independent Gaussian arm with different parameters. The results, depicted in Figure \ref{fig:arm}, indicate that the average regret will increase with the number of arms $K$ and this observed relationship demonstrates a nearly linear correlation. Such a pattern is consistent in both the risk-neutral case (\cite{auer2002finite}) and the risk-aware case (\cite{sani2012risk}). This trend holds true as the upper bounds of the expected regret are linear in the number of arms.}

\vspace{-10mm}
\begin{figure}[H]
\centering
      \begin{overpic}[width=0.7\textwidth]{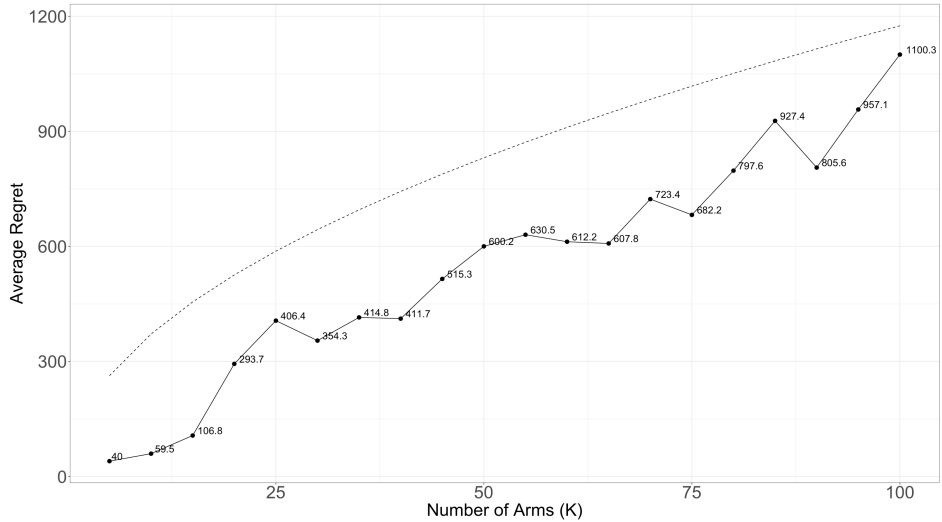}
      \end{overpic}
      \vspace{-10mm}
    \caption{\black{Average regret comparison averaged over 1000 runs for RALCB algorithms over different choices of $K$. In the graph, the solid line signifies the average regret derived from the simulation, while the dashed line represents the corresponding theoretical regret for different values of $K$.} \label{fig:arm}}
\end{figure}

\begin{remark}
    \black{Section \ref{section5} illustrates that our algorithm can be used to analyze the dependent case. However, analyzing correlations among arms is impossible in our current setting, where only one arm is pulled in each round. There is no way to exploit the correlation among arms without altering the underlying model setup (i.e. additional information regarding the prior structure and multiple arms are pulled simultaneously).}
\end{remark}
\black{Under the dependent scenario, we compare the performance of our algorithm over different choices of correlation coefficients (i.e., we use $\tau_{i,j}$ to denote the correlation coefficients between arm $i$ and arm $j$). As a starting point, we assume each pair of arms share the same correlation coefficients (i.e., $\tau_{i,j}=\tau$, $\forall i,j$). We are working on analyzing the performance changes as the correlations increase, where
$$\tau=(-0.5,-0.2,0,0.2,0.5,1).$$
Figure \ref{fig:10} illustrates the influence of correlation $\tau$ on the performance of the RALCB algorithm under different risk aversion $\widetilde\rho$ parameters.}

\begin{figure}[H]
\begin{subfigure}{.48\textwidth}
  \centering
  \includegraphics[width=1\linewidth]{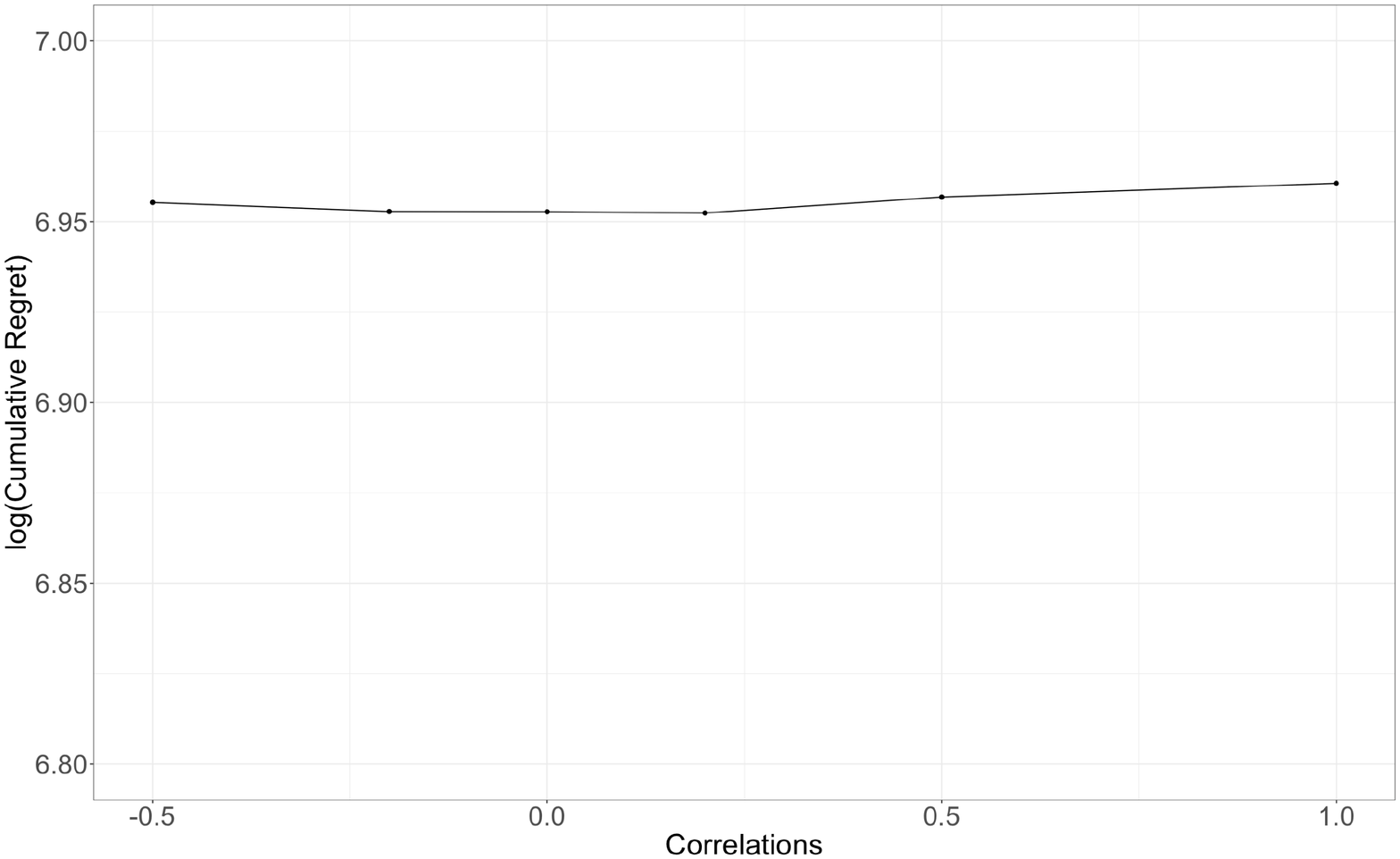}
  \caption{Cumulative regret over different $\tau$}
\end{subfigure}%
\begin{subfigure}{.48\textwidth}
  \centering
  \includegraphics[width=1\linewidth]{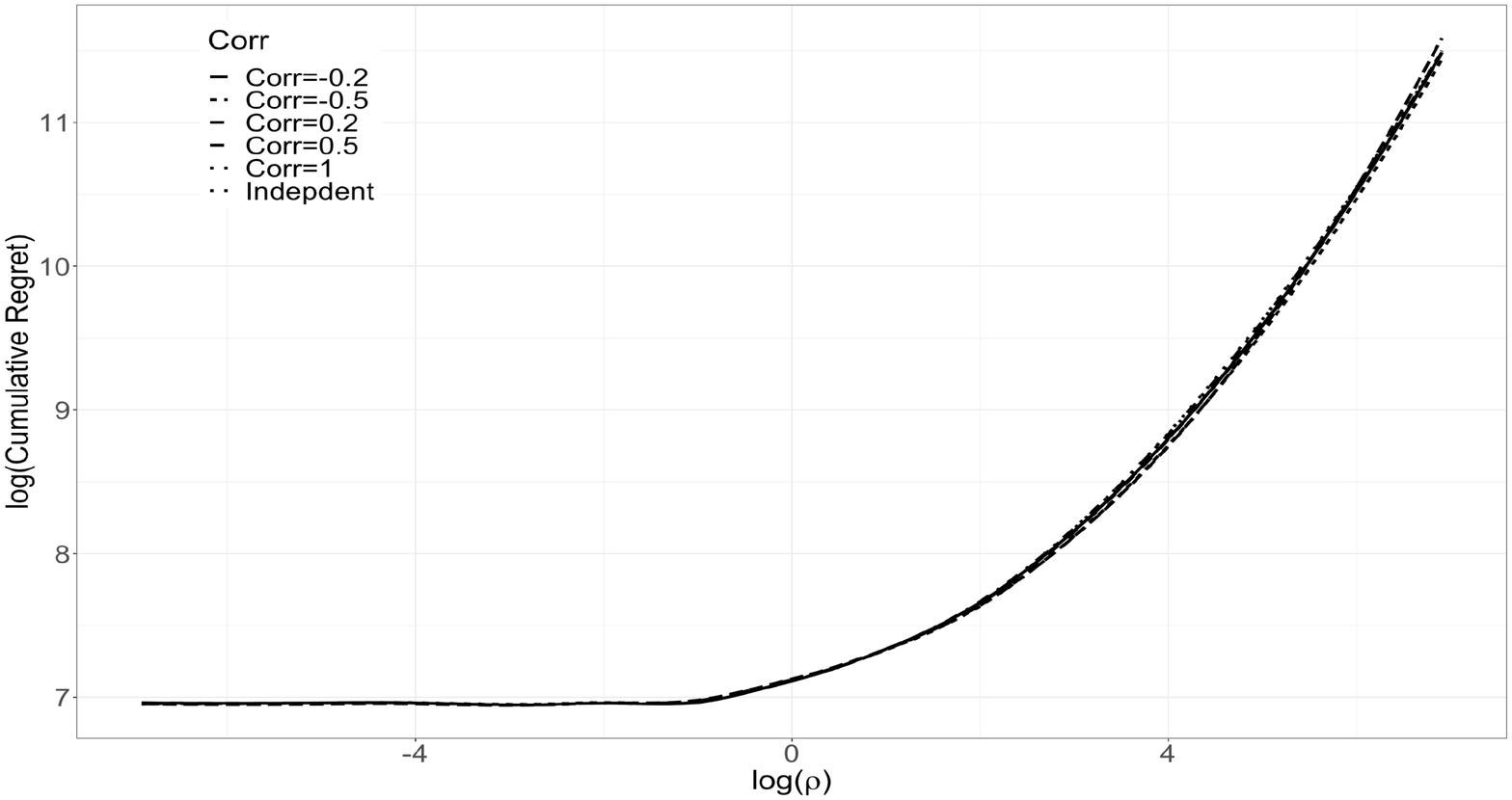}
  \caption{Cumulative regret comparison over different $\widetilde\rho$}
\end{subfigure}
\caption{\black{Cumulative regret comparison over different choices of $\tau$ and $\widetilde\rho$.}\label{fig:10}}
\end{figure}

 \black{Based on Figure \ref{fig:10}, it is evident that our algorithm consistently yields identical cumulative regret across various dependent scenarios. Without further information on the underlying structure or modifying the model setup, the examination of correlations among arms is unfeasible using our algorithm. The extension to examine correlated arms is left to future work.}\par
 
The numerical simulations discussed so far where all based on synthetic data sampled from standard classes of distributions. In the sequel, our goal is to illustrate the robustness of our algorithm by applying it to real-world data  with unknown distributions. To this end, we choose financial data due  to its wide availability. We emphasize, however, that financial data is publicly available, so that portfolio construction is typically not an application area for MAB algorithms, which inherently assume that the reward is only revealed for that arm that is drawn while the rewards of all other arms remain hidden. So the purpose of the following numerical experiments is solely the illustration of the robustness of our algorithm. We design experiments and evaluate the performance of our proposed algorithm (i.e., the RALCB algorithm) in portfolio selection to the performance of UCB algorithm, the $\varepsilon$-greedy algorithm and equally weighted (EW) strategy. In the experiment, we implement actual stock market data to test the performance of the above four algorithms. The dataset we use is S\&P500, which is the most frequent trading stock market data. We select 30 stocks from the S\&P500 stock pool based on size, age, and earning per share: AAPL, MSFT, UNH, JNJ, XOM, PNR, AIG, VNO, DISH, NWL, BK, STT, CL, HIG, ED, CMA, ALL, AMAT, BSX, MU, HAL, RCL, CCL, APA, HRL, AEE, FAST, TRMB, NWL, AFL. In our dataset, we have weekly observations of all 30 ETFs over the time period $1996/02/26$ to $2021/02/26$. 

Figure \ref{fig6sim} presents the change in portfolio wealth by following different algorithms over time. We notice that when the market turns down during the middle of the timeline, wealth for MAB-based algorithms drops dramatically. The main reason is that MAB-based algorithms must spend time learning new parameters because the underlying parameters estimated from the history are no longer valid. Overall, the RALCB algorithm achieves the best long-term performance and the highest wealth at the end.

\begin{figure}[H]
    \centering
    \includegraphics[width=0.9\textwidth]{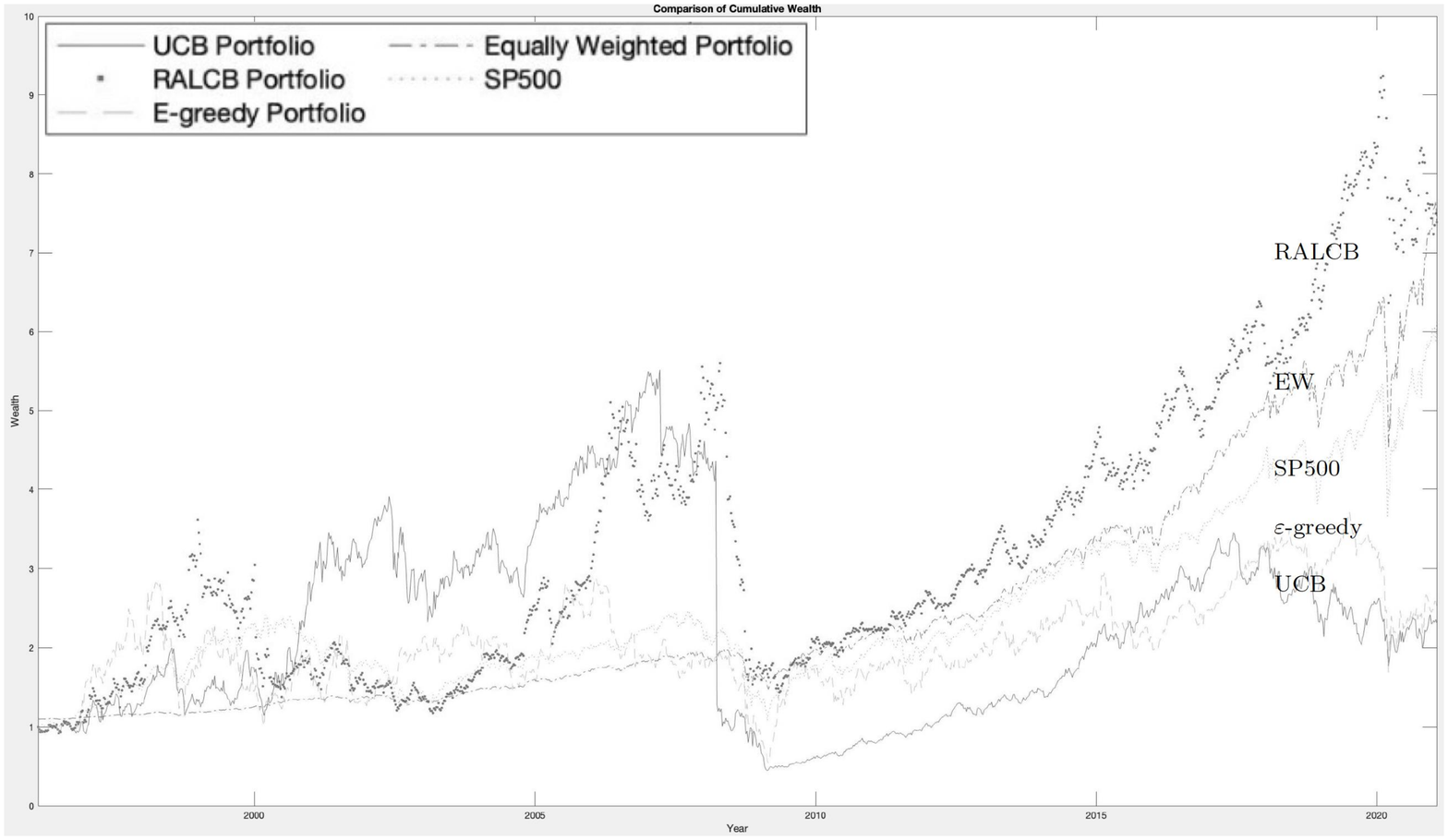}
    \vspace{-10mm}
    \caption{\black{The curves of cumulative wealth across the investment periods (1996-2021) for different algorithms.}\label{fig6sim}}
\end{figure}

To compare the performance of four investment algorithms, we use the standard criteria in finance proposed by \citet{zhu2019adaptive}: Cumulative Wealth (CW), Volatility (VO), Sharpe Ratio (SR) and Maximum Drawdown (MDD).\par

Table \ref{table11} summarizes portfolio performance as measured by the SR, CW, and VO for all benchmarks tested, with the bold values indicating the winners. The risk-free (Rf) rate, which is based on the historical average of the 10-Year Treasury Rate, is assumed to be 4.38\% yearly for generating the SR. Among all algorithms, RALCB algorithm is the one with the highest cumulative wealth. However, EW does a better job of risk-return balancing than RALCB. The main reason is that the RALCB algorithm only permits the investment of one stock per period, which naturally increases the portfolio's volatility.

\begin{table}[H]
    \centering
    \begin{tabular}{|c|c|c|c|c|}
  \hline
    1996-2021&UCB&RALCB&\black{$\varepsilon$-greedy}&EW\\ \hline
    CW&2.3579&\textbf{7.4988}&2.4741&7.3485\\ \hline
VO&	0.3777&0.2926&0.3151&\textbf{0.1921}\\ \hline
SR (Rf=4.38\%)&-0.0784&0.1480&-0.0490&\textbf{0.7330}\\ \hline
MDD&	0.9191&0.7432&0.3648&\textbf{0.2978}\\ \hline
\end{tabular}
    \caption{Performance of four different algorithms on real data. RALCB outperforms the other algorithms in terms of wealth accumulation. EW does a better job of risk-return balancing than RALCB. \label{table11}}
\end{table}

In our next application, the arms consist of the following S\&P 500 sector ETFs: GDX, IBB, ITB, IYE, IYR, KBE, KRE, OIH, SMH, VNQ, XHB, XLB, XLF, XLI, XLK, XLP, XLU, XLV, XLY, XME, XOP, XRT. In our dataset, we have daily observations of all 22 ETFs over the time period $2006/06/22$ to $2022/11/18$.

Figure \ref{etf} presents the change in portfolio wealth by following different algorithms over time. We notice that, by investing in ETFs, the RALCB algorithm achieves better long-term performance than the other algorithms.
\begin{figure}[H]
    \centering
    \includegraphics[width=0.9\textwidth]{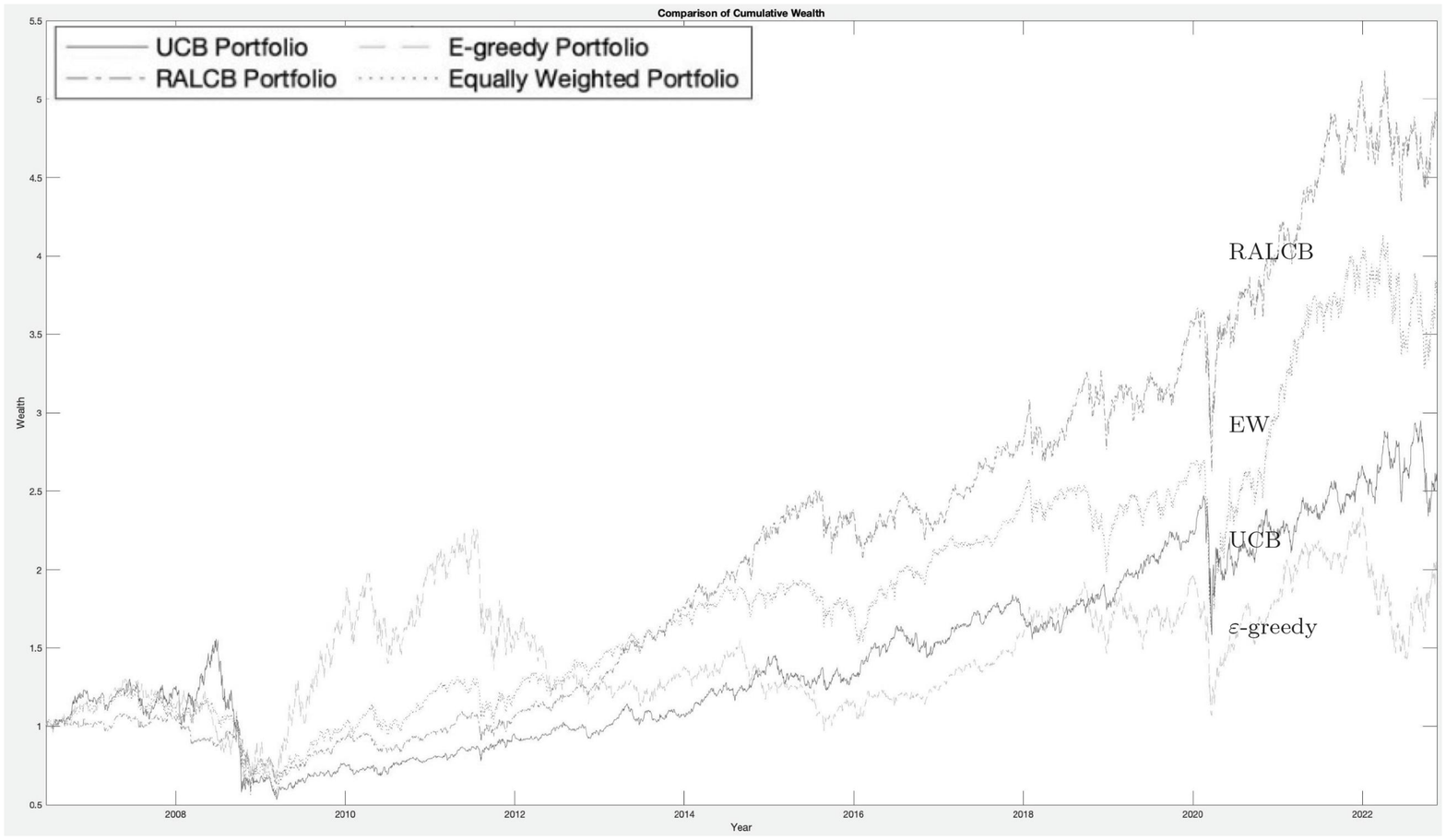}
    \vspace{-10mm}
    \caption{\black{The curves of cumulative wealth across the investment periods (2006-2022) for different algorithms.} \label{etf}}
\end{figure}

Again,  we use the standard criteria in finance proposed by \citet{zhu2019adaptive} to compare the performance of four investment algorithms:

\begin{table}[H]
    \centering
    \begin{tabular}{|c|c|c|c|c|}
  \hline
    2006-2022&UCB&RALCB&\black{$\varepsilon$-greedy}&EW\\ \hline
    CW&2.5629&\textbf{4.8620}&2.1932&3.7621\\ \hline
VO&	0.1004&\textbf{0.0824}&0.1444&0.1056\\ \hline
SR (Rf=4.38\%)&-0.2771&\textbf{-0.1396}&-0.2205&-0.2190\\ \hline
MDD&	0.6579&\textbf{0.4213}&0.6588&0.5616\\ \hline
\end{tabular}
    \caption{Performance of four different algorithms on ETFs investment. RALCB outperforms the other algorithms in terms of all financial criteria. \label{table22}}
\end{table}

Table \ref{table22} summarizes portfolio performance as measured by the SR, CW, and VO for all benchmarks tested, with the bold values indicating the winners. The risk-free (Rf) rate is again assumed to be 4.38\% yearly.


\section{Proofs}\label{section7}

This section consists of two parts. First, we derive the high-probability confidence bounds on the empirical mean-variance through an application of the Chernoff–Hoeffding inequality. In the second part, we report a complete analysis of the learning regret of the RALCB algorithm. Unless otherwise stated, all identities and inequalities between random variables are stated in $\mathbb{P}$-a.s. in the following proofs.

\subsection{Concentration Analysis}\label{appc}

To derive a high-probability bound on the expected pseudo regret, we will build high-probability confidence bounds on the empirical mean and variance through an application of the Chernoff–Hoeffding inequality. Recall that we assume that the rewards of each arm belong to the class of sub-Gaussian random variables. Therefore, we analyze the concentration of empirical mean and variance via sub-Gaussian and sub-exponential distributions. In this section, we let $\left\{X_{t}\right\}_{t\geq0}$ be i.i.d.~sub-Gaussian random variables with mean $\mu$, variance $\sigma^{2}$, and sub-Gaussianity parameter $\theta$.
 We start by providing a concentration bound for the empirical mean $\hat{\mu}_n:=\frac{1}{n} \sum_{t=1}^{n} X_{t}$. 

\begin{lemma}\label{Lem_7}
For every $n>0$ and $\delta \in (0,1)$, with probability at least $1-\delta$,
\begin{equation}
    |\hat{\mu}_n-\mu| \leq \theta \sqrt{\frac{2 \log (2/ \delta)}{n}}. \label{Eq_22}
\end{equation}
\end{lemma}

\begin{proof} 
According to Markov's inequality and \eqref{sub-Gaussian ieq}, we have for  $a\ge0$ and all $\lambda>0$,
\begin{align*}
\mathbb{P}[|\hat\mu_n-\mu|\ge a]\le  \mathbb{P}\Big[e^{-\lambda\sum_{t=1}^n(X_t-\mu)}\ge e^{\lambda na}\Big]+\mathbb{P}\Big[e^{\lambda\sum_{t=1}^n(X_t-\mu)}\ge e^{\lambda na}\Big]\le 2e^{\frac12\lambda^2\theta^2-\lambda na}\le 2 e^{-na^2/(2\theta^2)}.
\end{align*}
The result thus follows by taking $\delta=2e^{-na^2/(2\theta^2)}$.
\end{proof}

Now we derive a concentration inequality for the sample variance $\hat{\sigma}_n^2:=\frac{1}{n} \sum_{t=1}^{n} (X_{t}-\hat{\mu}_n)^2$. Its proof uses the notion of sub-exponential random variables. Recall that a random variable $Z$ is called \emph{sub-exponential with parameter $\gamma>0$} if $\mathbb{E}[Z]=0$ and $\mathbb{E}[e^{s Z}]\le e^{s^2\gamma^2/2}$ for all $s\in\mathbb{R}$ with $|s|\le1/\gamma$; see, e.g., \citet{wainwright2019high} and \citet{rigollet2017high}.

\begin{lemma}\label{Lem_8}
For every $n>0$ and $\delta \in (0,1)$, we have that with probability at least $1-\delta$,
\begin{equation}
    \left| \hat{\sigma}_n^2-\sigma^2+(\hat{\mu}_n-\mu)^2\right| \leq 32\theta^2 \max \left(\sqrt{\frac{\log (2/ \delta)}{n}},\frac{2 \log (2/ \delta)}{n}\right). \label{Eq_24}
\end{equation}
\end{lemma}

\begin{proof}
We may assume without loss of generality that $\mu=0$. Lemma 1.12 in \citet{rigollet2017high} states that $Z_t:=X_t^2$ is sub-exponential with parameter $16 \theta^{2}$. An application of Proposition 2.9 in  \citet{wainwright2019high} hence yields that 
\begin{equation}
  \mathbb{P}\bigg(\bigg|\frac{1}{n}\sum_{i=1}^{n}(X_i)^2-\sigma^2\bigg|>a\bigg) \leq 2 \exp \left(-\frac{n}{2}g(a)\right), \label{Eq_25}
\end{equation}
where $g(a):=\min(\frac{a^2}{512\theta^4},\frac{a}{32\theta^2})$. Moreover, 
$$
\begin{aligned}
\left|\frac{1}{n}\sum_{i=1}^{n}(X_i)^2-\sigma^2\right|=\left|\hat{\sigma}_n^2-\sigma^2+\hat{\mu}_n^2\right|.
\end{aligned}
$$
Therefore, we have,
\begin{equation}
    \mathbb{P}\left(\left|\hat{\sigma}_n^2-\sigma^2+\hat{\mu}_n^2\right|>a\right) \leq 2 \exp \left(-\frac{n}{2}g(a)\right). \label{Eq_26}
\end{equation}
By setting $\delta:=2 \exp \left(-\frac{n}{2}g(a)\right)$ and writing $a$ in terms of $\delta$ on the left side of the inequality \eqref{Eq_26}, we obtain:

\begin{equation}
     \mathbb{P}\left(\left|\hat{\sigma}_n^2-\sigma^2+\hat{\mu}_n^2\right|>32\theta^2 \max \left(\sqrt{\frac{\log (2/ \delta)}{n}},\frac{2 \log (2/ \delta)}{n}\right)\right) \leq \delta. \label{Eq_new1}
\end{equation}
Alternatively, with a probability of at least $1-\delta$,
$$ \left| \hat{\sigma}_n^2-\sigma^2+(\hat{\mu}_n-\mu)^2\right| \leq 32\theta^2 \max \left(\sqrt{\frac{\log (2/ \delta)}{n}},\frac{2 \log (2/ \delta)}{n}\right).$$
\end{proof}

In our analysis, we assume the rewards of each arm belong to the class of sub-Gaussian random variables with different parameters. With Lemma \ref{Lem_7} and Lemma \ref{Lem_8}, we can construct the following concentration inequality on the empirical mean-variance $\widehat{\mathrm{MV}}^{\rho}_{n}:=(1-\rho)\hat{\sigma}_{n}^{2}-\rho\hat{\mu}_{n}$, where the true mean-variance defined as $\mathrm{MV}^{\rho}:=(1-\rho)\sigma^{2}-\rho\mu$.

\begin{lemma}\label{Lem_11}

Then, for every $n>0$ and $\delta \in (0,1)$, we have that with probability at least $1-\delta$,
\begin{equation}
   \left|\widehat{\mathrm{MV}}^{\rho}_{n}-\mathrm{MV}^{\rho}\right| \leq \varphi \left(\frac{2\log (2/ \delta)}{n}\right),
\end{equation}
where, 
$$\varphi(x):=32\theta^2 \max\left(\sqrt{x/2},x\right)+\theta^2 x+\rho \theta \sqrt{x}/(1-\rho).$$

 \end{lemma}
\begin{proof}
\black{Let
$$A_{n}:=\left\{|\hat{\mu}_{n}-\mu| \leq \theta \sqrt{\frac{2 \log (2/ \delta)}{n}}\right\},$$
and
$$B_{n}:=\left\{\left| \hat{\sigma}_{n}^2-\sigma^2+(\hat{\mu}_{n}-\mu)^2\right| \leq 32\theta^2 \max \left(\sqrt{\frac{ \log (2/ \delta)}{n}},\frac{2 \log (2/ \delta)}{n}\right)\right\}.$$
By Lemma \ref{Lem_7} and \ref{Lem_8}, we have $\mathbb{P}(A)\geq 1-\delta$ and $\mathbb{P}(B)\geq 1-\delta$. Now, we define $$C_n:=\left\{\left|\widehat{\mathrm{MV}}^{\rho}_{n}-\mathrm{MV}^{\rho}\right| \leq \varphi \left(\frac{2\log (2/ \delta)}{n}\right)\right\}.$$ 
According to the triangle inequality, we have:
$$
\begin{aligned}
\left|\widehat{\mathrm{MV}}^{\rho}_{n}-\mathrm{MV}^{\rho}\right|&\leq (1-\rho)\left| \hat{\sigma}_{n}^2-\sigma^2\right|+\rho|\hat{\mu}_{n}-\mu| \\
&\leq(1-\rho)\left| \hat{\sigma}_{n}^2-\sigma^2+(\hat{\mu}_{n}-\mu)^2\right|+(1-\rho)(\hat{\mu}_{n}-\mu)^2+\rho|\hat{\mu}_{n}-\mu|.
\end{aligned}
$$
As a result, $A_{n}\cap B_{n} \subseteq C_{n}$. Therefore, we have:
$$\mathbb{P}\left[C_{n}^{c}\right]\leq \mathbb{P}\left[C_{n}^{c}\right]\leq \mathbb{P}\left[A_{n}^{c} \cup B_{n}^{c}\right]
\leq \mathbb{P}\left[A_{n}^{c}\right]+\mathbb{P}\left[B_{n}^{c}\right]\leq \delta.$$}
\end{proof}.
\subsection{Regret Analysis}\label{appd}
\subsubsection{Expected Regret}

Given the definition of regret in Equation ($\ref{Eq_5}$), \citet{sani2012risk} derived the following upper bound for the regret.  For $i^{\rho}_0 \in \argmin_{i\in \mathcal{A}} \textup{MV}^{\rho}_i$, $\widehat{\Delta}_{i}:=(1-\rho)\left(\hat{\sigma}_{i, n}^{2}-\sigma_{i^{\rho}_0}^{2}\right)-\rho\left(\hat{\mu}_{i, n}-\mu_{i^{\rho}_0}\right)$, and \black{$\widehat{\Gamma}_{i, h}:=\left|\hat{\mu}_{i, n}-\hat{\mu}_{h, n}\right|$}. \black{With probability 1}, the regret for a learning algorithm $\pi_t$ over $n$ rounds can be bounded as follows,
\begin{equation}
    \mathcal{R}_n(\pi) \leq \frac{1}{n} \sum_{i=1}^{K} T_{i, n} \widehat{\Delta}_{i}+\frac{1}{n^{2}} \sum_{i=1}^{K} \sum_{h \neq i} T_{i, n} T_{h, n} \widehat{\Gamma}_{i, h}^{2}.\label{Eq_7}
\end{equation}
\black{This result applies to any strategies addressing the mean-variance Multi-Armed Bandit (MAB) problem, including the RALCB algorithm.}
\black{The right-hand side of (\ref{Eq_7}) consists of two main components. The first term reflects regret as the difference in mean-variance between the best arm and the suboptimal arms. The second term represents exploration risk, being a weighted variance of the means. Excessive exploration increases regret, while consistently pulling a single arm eliminates exploration risk, though it may be suboptimal in terms of overall regret.}

\black{The upper bound \eqref{Eq_7} suggests that a bound on the pulls is sufficient to bound the regret. Subsequently, by computing expectations on both sides, as outlined in \cite{zhu2020thompson}, we obtain:
\begin{equation}
    \mathbb{E}[\mathcal{R}_n(\pi)]\leq\frac{1}{n}\sum_{i \in \mathcal{A}\setminus \mathcal{A}^{0}} \mathbb{E}[T_{i, n}](\Delta_{i}+2\Gamma_{i, \max }^{2})+\frac{5}{n}\sum_{i=1}^{K} \sigma_i^2,\label{Eq_10}
\end{equation}
where $\Delta_{i}:=(1-\rho)\left(\sigma_{i}^{2}-\sigma_{i^{\rho}_0}^{2}\right)-\rho\left(\mu_{i}-\mu_{i^{\rho}_0}\right)$, \black{$ \Gamma_{i, \max }:=\max \left\{\left|\mu_{i}-\mu_{h}\right|: h=1, \ldots, K\right\}$} and $\mathcal{A}^{0}$ is the set of optimal arms given by $\argmin_{i\in \mathcal{A}} \textup{MV}^{\rho}_i$.
This result also holds true for any strategies dealing with the mean-variance Multi-Armed Bandit (MAB) problem, including the RALCB algorithm. Equation \eqref{Eq_10} shows that to recover a bound on the expected regret, it suffices to bound the number of pulls of each suboptimal arm. Following an idea of \citet{sani2012risk}, we now derive an upper bound for the number of pulls in expectation.}

\begin{lemma}\label{Lem_4}
\black{Suppose that $i^{\rho}_0 \in \argmin_{i\in \mathcal{A}} \textup{MV}^{\rho}_i$. The expected number of pulls of any suboptimal arm $i \neq i^{\rho}_0$ in RALCB can be upper bounded by
\begin{equation}
    \mathbbm{E}\left[T_{i,n}\right] \leq \frac{8\log(n)}{\varphi^{-1}(\Delta_i/2)}+5 \label{Eq_11},
\end{equation}
where $n>0$ and $\varphi^{-1}$ is the inverse function of $\varphi$ and is given by
\begin{equation*}
\varphi^{-1}(x) = \left\{
        \begin{array}{ll}
             \left(\frac{-(\rho\theta+16\sqrt{2}\theta^2)+\sqrt{(\rho\theta+16\sqrt{2}\theta^2)^2+4\theta^2 x}}{2\theta^2}\right)^2 & \quad x \in [0,\frac{17}{2}\theta^2+\frac{\sqrt{2}}{2}\rho\theta), \\
            \left(\frac{-\rho\theta+\sqrt{\rho^2\theta^2+132\theta^2x}}{66\theta^2}\right)^2 & \quad x \in [\frac{17}{2}\theta^2+\frac{\sqrt{2}}{2}\rho\theta,\infty).
        \end{array}
    \right.
\end{equation*}}
\end{lemma}

\begin{proof}
  Let $i$ be a suboptimal arm and  $\ell$ be an arbitrary positive integer. 
$$
\begin{aligned}
T_{i,n}&=  1+\sum_{t=K+1}^n\mathbbm{1}_{\left\{\pi_t=i\right\}}\\
&\leq \ell+\sum_{t=\ell+1}^n \mathbbm{1}_{\left\{\pi_t=i, T_{i,t} \geq \ell\right\}} \\
& = \ell+\sum_{t=\ell+1}^n \mathbbm{1}_{\left\{V^{\textup{RALCB}}_{i,t} \leq V^{\textup{RALCB}}_{i^{\rho}_0,t}, T_{i,t} \geq \ell\right\}}.
\end{aligned}
$$
In the above, with setting $\delta:=2/t^4$:

$$
\begin{aligned}
\mathbbm{1}_{\left\{V^{\textup{RALCB}}_{i,t} \leq V^{\textup{RALCB}}_{i^{\rho}_0,t}, T_{i,t} \geq \ell\right\}}&= \mathbbm{1}_{\left\{\widehat{\mathrm{MV}}^{\rho}_{i,T_{i,t}}-\varphi \left(\frac{8\log(t)}{T_{i,t}}\right)\leq \widehat{\mathrm{MV}}^{\rho}_{i^{\rho}_0, T_{i^{\rho}_0,t}}-\varphi \left(\frac{8\log(t)}{T_{i^{\rho}_0,t}}\right), T_{i,t} \geq \ell\right\}}\\
&\leq \mathbbm{1}_{\left\{\min_{\ell \leq m\leq t}\widehat{\mathrm{MV}}^{\rho}_{i,m}-\varphi \left(\frac{8\log(t)}{m}\right)\leq \max_{0<h\leq t} \widehat{\mathrm{MV}}^{\rho}_{i^{\rho}_0, h}-\varphi \left(\frac{8\log(t)}{h}\right)\right\}}\\
&\leq \sum_{h=1}^{t} \sum_{m=\ell}^{t} \mathbbm{1}_{\left\{\widehat{\mathrm{MV}}^{\rho}_{i,m}-\varphi \left(\frac{8\log(t)}{m}\right)\leq \widehat{\mathrm{MV}}^{\rho}_{i^{\rho}_0, h}-\varphi \left(\frac{8\log(t)}{h}\right)\right\}}
\end{aligned}
$$

Consider that $\widehat{\mathrm{MV}}^{\rho}_{i,m}-\varphi \left(\frac{8\log(t)}{m}\right)- \widehat{\mathrm{MV}}^{\rho}_{i^{\rho}_0, h}+\varphi \left(\frac{8\log(t)}{h}\right)$ can be expressed as the sum of three terms:
$$
\begin{aligned}
&\widehat{\mathrm{MV}}^{\rho}_{i,m}-\varphi \left(\frac{8\log(t)}{m}\right)-\widehat{\mathrm{MV}}^{\rho}_{i^{\rho}_0, h}+\varphi \left(\frac{8\log(t)}{h}\right)\\
= & \left(\widehat{\mathrm{MV}}^{\rho}_{i,m}+\varphi \left(\frac{8\log(t)}{m}\right)-\mathrm{MV}^{\rho}_{i}\right)-\left(\widehat{\mathrm{MV}}^{\rho}_{i^{\rho}_0,h}-\varphi \left(\frac{8\log(t)}{h}\right)-\mathrm{MV}^{\rho}_{i^{\rho}_0}\right) \\
& +\left(\mathrm{MV}^{\rho}_{i}-\mathrm{MV}^{\rho}_{i^{\rho}_0}-2 \varphi \left(\frac{8\log(t)}{m}\right)\right) .
\end{aligned}
$$

For $m\ge\ell$ and $\ell=\left\lceil\frac{8\log(t)}{\varphi^{-1}(\Delta_i/2)}\right\rceil$, the last term above is positive. Therefore,
$$
\begin{aligned}
T_{i,n} &\leq  \ell+\sum_{t=\ell+1}^n \sum_{h=1}^{t} \sum_{m=\ell}^{t}\mathbbm{1}_{\left\{\widehat{\mathrm{MV}}^{\rho}_{i,m}+\varphi \left(\frac{8\log(t)}{m}\right)-\mathrm{MV}^{\rho}_{i}\leq 0\right\}}+\sum_{t=\ell+1}^n \sum_{h=1}^{t} \sum_{m=\ell}^{t} \mathbbm{1}_{\left\{\widehat{\mathrm{MV}}^{\rho}_{i^{\rho}_0,h}-\varphi \left(\frac{8\log(t)}{h}\right)-\mathrm{MV}^{\rho}_{i^{\rho}_0}\geq 0\right\}}
\end{aligned}
$$
Set $\delta:=2/t^4$, provided that $t\ge2$. Note that or each arm $i$, $X^{i}_{1},...,X^{i}_{m}$ are i.i.d.~sub-Gaussian random variables with mean $\mu_{i}$ and variance $\sigma^2_{i}$. Hence we may apply Lemma \ref{Lem_11}, which yields that
$$
\mathbbm{P}\left[\widehat{\mathrm{MV}}^{\rho}_{i,m}+\varphi \left(\frac{8 \log (t)}{m}\right)-\mathrm{MV}^{\rho}_{i}\leq 0\right] \leq 2/t^4.
$$
Similarly, we have
$$
\mathbbm{P}\left[\widehat{\mathrm{MV}}^{\rho}_{i^{\rho}_0,h}-\varphi \left(\frac{8 \log (t)}{h}\right)-\mathrm{MV}^{\rho}_{i^{\rho}_0}\geq 0\right] \leq 2/t^4.
$$
Hence,
$$
\begin{aligned}
\mathbbm{E}\left[T_{i,n}\right] \leq & \ell+ 2\sum_{t=\ell+1}^n \sum_{h=1}^{t} \sum_{m=\ell}^{t} 2/t^{4} \\
\leq & \frac{8\log(n)}{\varphi^{-1}(\Delta_i/2)}+1+4 \int_{\ell}^{\infty} t^{-2} d t \\
= & \frac{8\log(n)}{\varphi^{-1}(\Delta_i/2)}+1+4 \ell^{-1} \\
\leq & \frac{8\log(n)}{\varphi^{-1}(\Delta_i/2)}+5 .
\end{aligned}
$$
The last inequality holds because each arm has at least been pulled once.
\end{proof}

\subsubsection{High Probability Regret}
\black{Consider $X^{i}_{1},...,X^{i}_{t}$ are i.i.d.~sub-Gaussian random variables with mean $\mu_{i}$ and variance $\sigma_{i}^{2}$, as well as empirical mean $\hat{\mu}_{i, t}$ and variance $\hat{\sigma}_{i, t}^{2}$. Take $\theta$ such that $\mathbb{E}_{\mathbb{P}}[e^{a(X-\mathbb{E}_{\mathbb{P}}[X])}]\leq e^{\frac{a^2 \theta^2 }{2}}$ holds for all $a \in \mathbbm{R}$ and $i \in \mathcal{A}$. The empirical and true mean-variance are defined as $\widehat{\mathrm{MV}}^{\rho}_{i, s}:=(1-\rho)\hat{\sigma}_{i, s}^{2}-\rho\hat{\mu}_{i, s}$ and $\mathrm{MV}^{\rho}_{i}:=(1-\rho)\sigma_{i}^{2}-\rho\mu_{i}$ respectively.}
\begin{lemma}\label{Lem_11_2}
\black{For $t>0$ and $\delta \in(0,1 / 2 t K)$, we have:
$$
\mathbb{P}\left[\max_{\substack{i=1, \ldots, K \\ s=1, \ldots, t}}\left|\widehat{\mathrm{MV}}_{i, s}^\rho-\mathrm{MV}_i^\rho\right| \geq \varphi\left(\frac{2 \log (2 / \delta)}{s}\right)\right] \leq 2 t K \delta.
$$}
\end{lemma}

\begin{proof}
   Let
$$
A_{i, s}:=\left\{\left|\hat{\mu}_{i, s}-\mu_i\right| \leq \theta \sqrt{\frac{2 \log (2 / \delta)}{s}}\right\}
$$
and
$$
B_{i, s}:=\left\{\left|\hat{\sigma}_{i, s}^2-\sigma_i^2+\left(\hat{\mu}_{i, s}-\mu_i\right)^2\right| \leq \black{32} \theta^2 \max \left(\sqrt{\frac{\log (2 / \delta)}{s}}, \frac{2 \log (2 / \delta)}{s}\right)\right\} .
$$
By Lemma \ref{Lem_7} and \ref{Lem_8}, we have $\mathbb{P}\left(A_{i, s}\right) \geq 1-\delta$ and $\mathbb{P}\left(B_{i, s}\right) \geq 1-\delta$. Now, we define
$$
C_{i, s}:=\left\{\left|\widehat{\mathrm{MV}}_{i, s}^\rho-\mathrm{MV}_i^\rho\right| \leq \varphi\left(\frac{2 \log (2 / \delta)}{s}\right)\right\} .
$$

According to the triangle inequality, we have for any $i$ and $t$ :
$$
\begin{aligned}
\left|\widehat{\mathrm{MV}}_{i, s}^\rho-\mathrm{MV}_i^\rho\right| & \leq(1-\rho)\left|\hat{\sigma}_{i, s}^2-\sigma_i^2\right|+\rho\left|\hat{\mu}_{i, s}-\mu_i\right| \\
& \leq(1-\rho)\left|\hat{\sigma}_{i, s}^2-\sigma_i^2+\left(\hat{\mu}_{i, s}-\mu_i\right)^2\right|+(1-\rho)\left(\hat{\mu}_{i, s}-\mu_i\right)^2+\rho\left|\hat{\mu}_{i, s}-\mu_i\right| .
\end{aligned}
$$

As a result, for any $i$ and $t, A_{i, s} \cap B_{i, s} \subseteq C_{i, s}$. Through the union bound, we have:
$$
\mathbb{P}\left[\bigcup_{i=1}^K \bigcup_{s=1}^t C_{i, s}^c\right] \leq \sum_{i, s} \mathbb{P}\left[C_{i, s}^c\right] \leq \sum_{i, s} \mathbb{P}\left[A_{i, s}^c \cup B_{i, s}^c\right] \leq \sum_{i, s}\left(\mathbb{P}\left[A_{i, s}^c\right]+\mathbb{P}\left[B_{i, s}^c\right]\right) \leq 2 t K \delta .
$$ 
\end{proof}

Now, we can turn to the proof of Theorem \ref{Thm_2}.

\begin{proof}[Proof of Theorem \ref{Thm_2}]
First, we derive the claimed upper bound for the regret. By the definitions in Section \ref{appd}, we have:
$$\widehat{\Delta}_{i}=\Delta_i+(1-\rho)\left(\hat{\sigma}_{i, T_{i,t}}^{2}-\sigma_{i}^{2}\right)-\rho\left(\hat{\mu}_{i, T_{i,t}}-\mu_{i}\right),$$
and
$$\black{\left|\widehat{\Gamma}_{i, h}\right|}=\left|\Gamma_{i, h}-\mu_i+\mu_h+\hat{\mu}_{i, T_{i,t}}-\hat{\mu}_{h, T_{h,t}}\right|.$$

\black{According to Lemma \ref{Lem_11_2}, by setting $\delta=2/t^4$ and given $s\leq t$, we can construct a high-probability event:}\par
 \scalebox{0.8}{$
\mathcal{E}_{t}:=\left\{ \forall i=1,...,K,\; \forall s=1,...,t,\; |\hat{\mu}_{i,s}-\mu_i| \leq \theta \sqrt{\frac{8 \log (t)}{s}},\; \left| \hat{\sigma}_{i,s}^2-\sigma_{i}^2+(\hat{\mu}_{i,s}-\mu_i)^2\right| \leq 32\theta^2 \max \left(\sqrt{\frac{ 4\log (t)}{s}},\frac{8 \log (t)}{s}\right)\right\}.
$}

Through the union bound over arms and periods, we have $\mathbb{P}\left[\mathcal{E}_{t}^{c}\right]\leq K/t^3$ where \black{$t \in (K^{1/3},\infty)$}.
\black{With probability at least $1-K/t^3$}, by Lemma \ref{Lem_7}, we have $\left|\widehat{\Gamma}_{i, h}\right|\leq \left|\Gamma_{i, h}\right|+\theta \sqrt{8\log (t)/\black{T_{i,t}}}+\theta \sqrt{8\log (t)/T_{h,t}}$. By Lemma \ref{Lem_11_2}, \black{with probability at least $1-K/t^3$}, we also have $\widehat{\Delta}_{i}\leq \Delta_i + \varphi\left(8\log (t)/T_{i,t}\right)$. According to \eqref{Eq_7},
\begin{equation*}
\begin{aligned} 
 \mathcal{R}_t(\pi) &\leq \frac{1}{t} \sum_{i=1}^{K} T_{i, t} \widehat{\Delta}_{i}+\frac{1}{t^{2}} \sum_{i=1}^{K} \sum_{h \neq i} T_{i, t} T_{h, t} \widehat{\Gamma}_{i, h}^{2} \\
 &\leq \frac{1}{t} \sum_{i=1}^{K} T_{i, t} \left(\Delta_{i}+ \varphi\left(\frac{8\log (t)}{T_{i,t}}\right)\right)\\
 &+\frac{1}{t^{2}} \sum_{i=1}^{K} \sum_{h \neq i} T_{i, t} T_{h, t} \left( \left|\Gamma_{i, h}\right|+\theta \sqrt{\frac{8\black{\log (t)}}{T_{i,t}}}+\theta \sqrt{\frac{8\log (t)}{T_{h,t}}}\right)^{2} \\
 &\leq \frac{1}{t}\sum_{i=1}^{K} T_{i, t} \Delta_{i} +\frac{1}{t}\sum_{i=1}^{K} T_{i, t} \varphi\left(\frac{8\log (t)}{T_{i,t}}\right)+\frac{2}{t^{2}} \sum_{i=1}^{K} \sum_{h \neq i} T_{i, t} T_{h, t} \Gamma_{i, h}^{2} \\
 &+\frac{8\theta^2\sqrt{2}}{t^{2}} \sum_{i=1}^{K} \sum_{h \neq i}T_{h,t}\log(t)+\frac{8\theta^2\sqrt{2}}{t^{2}} \sum_{i=1}^{K} \sum_{i \neq i}T_{h,t}\log(t) \\
 &\leq \frac{1}{t}\sum_{i=1}^{K} T_{i, t} \Delta_{i}+\frac{2}{t^{2}} \sum_{i=1}^{K} \sum_{h \neq i} T_{i, t} T_{h, t} \Gamma_{i, h}^{2}\\
 &+ 32(1-\rho)\theta\max\left(\sqrt{\frac{8K\log (t)}{t}},\frac{4K\log (t)}{t}\right)+4(1+4\sqrt{2})(1-\rho)\frac{K\theta\log (t)}{t}+\rho\theta\sqrt{\frac{8K\log (t)}{t}}\\
 &\leq \frac{1}{t}\sum_{i=1}^{K} T_{i, t} \Delta_{i}+\frac{2}{t^{2}} \sum_{i=1}^{K} \sum_{h \neq i} T_{i, t} T_{h, t} \Gamma_{i, h}^{2}+\frac{\varphi\left(\frac{8K\log (t)}{t}\right)}{\theta}+\frac{16\sqrt{2}(1-\rho)K\theta\log (t)}{t},
\end{aligned}
\end{equation*}

with probability $1-K/t^3$. In the next to last passage we used Jensen’s inequality for concave functions and \black{leveraged the facts that $K-1<K$ and $\sum_i T_{i,t}\leq t$}. 
Based on \eqref{Eq_10} and Lemma \ref{Lem_4}, we can obtain:
\begin{equation*}
\begin{aligned} 
\frac{1}{t}\sum_{i=1}^{K} T_{i, t} \Delta_{i}+\frac{2}{t^{2}} \sum_{i=1}^{K} \sum_{h \neq i} T_{i, t} T_{h, t} \Gamma_{i, h}^{2}&\leq\frac{1}{t}\sum_{i \in \mathcal{A}\setminus \mathcal{A}^{0}} T_{i, t}(\Delta_{i}+2\Gamma_{i, \max }^{2})\\
&\leq\frac{1}{t}\sum_{i \in \mathcal{A}\setminus \mathcal{A}^{0}} \left(\frac{8\log(t)}{\varphi^{-1}(\Delta_i/2)}+5\right)(\Delta_{i}+2\Gamma_{i, \max }^{2}),
\end{aligned}
\end{equation*}
\black{with probability at least $1-K/t^3$}. Combining the above results gives us the stated regret bound.%
\end{proof}

\bibliographystyle{apa} 
\bibliography{reference.bib}

\end{document}